%% file: Thom.tex
\documentclass[10pt]{amsart}
\usepackage{amsmath,amsthm,amsfonts,amssymb,amscd,flafter,pinlabel}
\usepackage{mathtools}
\usepackage[mathscr]{eucal}
\usepackage{graphics}
 \usepackage[all]{xy}
\usepackage{caption, subcaption}
\usepackage[usenames,dvipsnames]{color}
\usepackage{graphicx}

\usepackage{pgf}
\usepackage{tikz}
\usetikzlibrary{arrows,automata}
\usepackage[latin1]{inputenc}

\usepackage[colorlinks=true, urlcolor=NavyBlue, linkcolor=NavyBlue, citecolor=NavyBlue, pdftitle={Bridge trisections in CP2 and the Thom conjecture}, pdfauthor={Peter Lambert-Cole}, pdfsubject={Thom conjecture}, pdfkeywords={Thom conjecture, 4-manifolds, bridge trisections, minimal genus}]{hyperref}

\newtheorem{theorem}{Theorem}[section]
\newtheorem{corollary}[theorem]{Corollary}

\newtheorem{question}[theorem]{Question}
\newtheorem{proposition}[theorem]{Proposition}

\newtheorem{lemma}[theorem]{Lemma}

\theoremstyle{definition}
\newtheorem{definition}[theorem]{Definition}

\theoremstyle{definition}
\newtheorem{remark}[theorem]{Remark}

\newcommand{\R}{\ensuremath{\mathbb{R}}}

\newcommand\alphas{\mbox{\boldmath$\alpha$}}
\newcommand\betas{\mbox{\boldmath$\beta$}}
\newcommand\gammas{\mbox{\boldmath$\gamma$}}

\newcommand\SH{\mathcal{H}}

\newcommand{\CC}{\mathbb{C}}
\newcommand{\RR}{\mathbb{R}}
\newcommand{\del}{\partial}

\newcommand{\ZZ}{\mathbb{Z}}

\newcommand{\cD}{\mathcal{D}}
\newcommand{\cL}{\mathcal{L}}

\newcommand{\cF}{\mathcal{F}}
\newcommand{\cT}{\mathcal{T}}

\newcommand{\cA}{\mathcal{A}}
\newcommand{\cC}{\mathcal{C}}
\newcommand{\cK}{\mathcal{K}}
\newcommand{\cS}{\mathcal{S}}
\newcommand{\cB}{\mathcal{B}}

\newcommand{\DD}{\mathbb{D}}

\newcommand{\Khat}{\widehat{K}}

\newcommand{\CP}{\mathbb{CP}}

\begin{document}

\title[{Bridge trisections in $\CP^2$ and the Thom conjecture}]{Bridge trisections in $\CP^2$ and the Thom conjecture \\ (with Corrigendum)}

\author[P. Lambert-Cole]{Peter Lambert-Cole}
\address{School of Mathematics \\ Georgia Institute of Technology}
\email{plc@math.gatech.edu}
\urladdr{\href{http://www.math.gatech.edu/~plambertcole3}{http://www.math.gatech.edu/~plambertcole3}}

\keywords{Thom conjecture, 4-manifolds, bridge trisections, minimal genus}
\subjclass[2010]{57R17, 57R40}
\maketitle


\begin{abstract}

In this paper, we develop new techniques for understanding surfaces in $\CP^2$ via bridge trisections.  Trisections are a novel approach to smooth 4-manifold topology, introduced by Gay and Kirby, that provide an avenue to apply 3-dimensional tools to 4-dimensional problems.  Meier and Zupan subsequently developed the theory of bridge trisections for smoothly embedded surfaces in 4-manifolds.  The main application of these techniques is a new proof of the Thom conjecture, which posits that algebraic curves in $\CP^2$ have minimal genus among all smoothly embedded, oriented surfaces in their homology class.  This new proof is notable as it completely avoids any gauge theory or pseudoholomorphic curve techniques. \\

\noindent {\bf Corrigendum}: This paper contains a fatal error in the proof of Theorem \ref{thrm:Thom}, which is the headline result of the paper.  The error is localized to Section 6 and is described in a Corrigendum at the end of this updated version.  The remaining results in Sections 1 through 5 remain valid.

\end{abstract}

\input{introduction}

\input{diagram}

\input{transverse}


\input{braiding}

\input{rbi}

\input{adjunction}


\bibliographystyle{alpha}
\nocite{*}
\bibliography{References}


\end{document}

%% file: introduction.tex
\section{Introduction}

A trisection of a smooth, oriented 4-manifold $X$ is a particular decomposition into three elementary pieces.  It is a 4-dimensional analogue of Heegaard splittings, where a 3-manifold is bisected into two handlebodies glued together along their boundary.  Recently, Meier and Zupan have extended this perspective to bridge trisections of knotted surfaces in 4-manifolds \cite{MZ-Bridge-Trisection,MZ-GBT}.  Bridge trisections are a 4-dimensional analogue of bridge position for links in a 3-manifold.  A bridge splitting of a link $L$ is a decomposition of $L$ into a pair of trivial tangles.  Similarly, a bridge trisection of a knotted surface is a decomposition into a triple of trivial disk tangles.

The projective plane $\CP^2$ admits a genus 1 trisection well-adapted to its complex and toric geometry.  This geometric compatibility makes it possible to apply topological methods in 3-dimensional contact geometry to the study of smooth surfaces in $\CP^2$.  In this paper, we introduce several new techniques, concepts and results regarding bridge trisections and diagrams for bridge trisections.  The main application is a new proof of the Thom conjecture.

\begin{theorem}[Thom Conjecture \cite{KM-Thom}]
\label{thrm:Thom}
Let $\cK$ be a smoothly embedded, oriented, connected surface in $\CP^2$ of degree $d > 0$.  Then
\[g(\cK) \geq \frac{1}{2}(d-1)(d-2).\]
\end{theorem}

The conjecture was originally proved by Kronheimer and Mrowka using Seiberg-Witten gauge theory \cite{KM-Thom}.  A generalization to K\"ahler manifolds, that complex curves minimize genus, was subsequently proved by Morgan, Szab\'o and Taubes \cite{MST}.  Alternate proofs of the original conjecture in $\CP^2$ were subsequently given by Ozsv\'ath and Szab\'o using Heegaard Floer homology \cite{OS-HF-Thom-1} and by Strle \cite{Strle}.

The novelty of this trisections proof is that we completely avoid any gauge theory or pseudoholomorphic curve techniques.  In particular, using the techniques introduced in this paper, we can reduce the adjunction inequality to the {\it ribbon-Bennequin inequality}. 

\begin{theorem}[Ribbon-Bennequin inequality]
\label{thrm:rbi}
Let $L$ be a transverse link in $(S^3, \xi_{std})$ and let $F$ be a ribbon surface bounded by $L$.  Then
\[sl(L) \leq - \chi(F).\]
\end{theorem}

This is the ribbon surface equivalent of the well-known {\it slice-Bennequin inequality}, which was conjectured by Bennequin \cite{Bennequin} and first proved by Rudolph \cite{Rudolph-Slice-Bennequin}.  Rudolph proved that slice-Bennequin is equivalent to the {\it Local Thom conjecture} on the slice genera of torus knots.

\begin{theorem}[Local Thom Conjecture \cite{KM-Milnor}]
\label{thrm:local-thom}
The slice genus of $T(p,q)$ is $\frac{1}{2}(p - 1)(q - 1)$.
\end{theorem}

The Local Thom conjecture was also proved by Kronheimer and Mrowka \cite{KM-Milnor}, using Donaldson invariants.  Later, Rasmussen \cite{Rasmussen-Milnor} introduced a concordance invariant in Khovanov homology and { gave a combinatorial proof of} the Local Thom conjecture\footnote{Theorem \ref{thrm:local-thom} is referred to as the {\it Milnor conjecture} in \cite{Rasmussen-Milnor}. }.  Shumakovitch \cite{Shumakovitch} then noted that the slice-Bennequin inequality is also easily deduced using the $s$-invariant.  Consequently, as the slice-Bennequin inequality trivially implies the ribbon-Bennequin inequality, our proof of Theorem \ref{thrm:Thom} reduces to Rasmussen's combinatorial proof and avoids gauge theory.

In effect, the proof uses the Local Thom conjecture to deduce the (global) Thom conjecture.  This reverses the standard approach, such as in \cite{KM-Milnor,Lisca-Matic}, whereby a global adjunction inequality is used to deduce information about slice genera.  As the Local Thom conjecture can be easily deduced from the global version, we have the following corollary.

\begin{theorem}
The Thom conjecture is equivalent to the Local Thom conjecture.
\end{theorem}

Moreover, based on the way in which a trisection decomposition shuffles the topology of a 4-manifold, applying these ideas in larger 4-manifolds is a promising strategy to recover and extend adjunction inequalities.  For example, the most general version of Theorem \ref{thrm:Thom} is the {\it symplectic Thom conjecture}, proved by Ozsv\'ath and Szab\'o \cite{OS-Symplectic-Thom}, which posits that symplectic surfaces in a symplectic 4-manifold minimize genus in their homology class.  There has been some progress in trisections of K\"ahler surfaces and curves on them \cite{MZ-GBT,LC-Meier}.  However, the interaction between trisections and K\"ahler or symplectic geometry remains a deep open question.

Finally, as mentioned above, our proof reduces to the ribbon-Bennequin inequality.  Interestingly, this is purely a 3-dimensional statement, as opposed to the 4-dimensional slice-Bennequin inequality.  It would be extremely interesting to give a proof using only 3-dimensional contact geometry, perhaps by reducing it to the Bennequin-Eliashberg inequality.  Given the deep geometric connection between tight contact structures and the adjunction inequality, this would complete an extremely satisfying proof of the Thom conjecture.

\subsection{Trisections}
\label{sub:trisections}

Let $X$ be a smooth, closed, oriented 4-manifold.  A $(d + 1)$-dimensional {\it 1-handlebody} of genus $g$ is the compact $(d+1)$-manifold $\natural^g (S^1 \times D^d)$.

\begin{definition} [\cite{Gay-Kirby-Trisections}]
A $(g;k_1,k_2,k_3)$-{\it trisection} $\cT$ of $X$ is a decomposition $X = X_1 \cup X_2 \cup X_3$ such that
\begin{enumerate}
\item Each $X_{\lambda}$ is a 4-dimensional 1-handlebody of genus $k_{\lambda}$, 
\item Each $H_{\lambda} = X_{\lambda-1} \cap X_{\lambda}$ is a 3-dimensional 1-handlebody of genus $g$, and
\item $\Sigma = X_1 \cap X_2 \cap X_3$ is a closed, oriented surface of genus $g$
\end{enumerate}
If $k_1 = k_2 = k_3 = k$, we call $\cT$ a $(g,k)$-trisection of $X$.
\end{definition}

The orientation of $X$ induces orientations on each $X_{\lambda}$ and each boundary $Y_{\lambda} \coloneqq \del X_{\lambda}$.  We choose to orient $H_{\lambda}$ by viewing it as a submanifold of $Y_{\lambda}$.  As a result, the orientation induced on $\Sigma$ as the boundary of $H_{\lambda}$ is independent of $\lambda$.  Moreover, we have that $Y_{\lambda} = H_{\lambda} \cup_{\Sigma} - H_{\lambda+1}$ as oriented manifolds.  

The {\it spine} of a trisection $\cT$ is the union $H_{1} \cup H_{2} \cup H_{3}$.  The spine uniquely determines the trisection $\cT$ and can be encoded by a {\it trisection diagram} $(\alphas,\betas,\gammas)$ consisting of three cut systems for $\Sigma$.  A cut system for $\Sigma$ consists of $g$ disjoint, simple closed curves whose complement in $\Sigma$ is a planar surface.  Each cut system corresponds to one handlebody $H_{\lambda}$.  The spine is constructed by attaching 2-handles along each curve in the cut system, followed by a single 3-handle for each handlebody $H_{\lambda}$.  If $X$ admits a $(g;k_1,k_2,k_3)$-trisection, then
\[\chi(X) = 2 + g - k_1 - k_2 - k_3.\]

\subsection{Trisection of $\CP^2$}

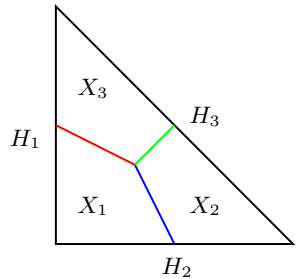
\begin{figure}
\centering
\labellist
	\small\hair 2pt
	\pinlabel $X_1$ at 80 80
	\pinlabel $X_2$ at 180 80
	\pinlabel $X_3$ at 80 185
	\pinlabel $H_{1}$ at 20 140
	\pinlabel $H_{2}$ at 155 25
	\pinlabel $H_{3}$ at 180 160
\endlabellist
\includegraphics[width=.28\textwidth]{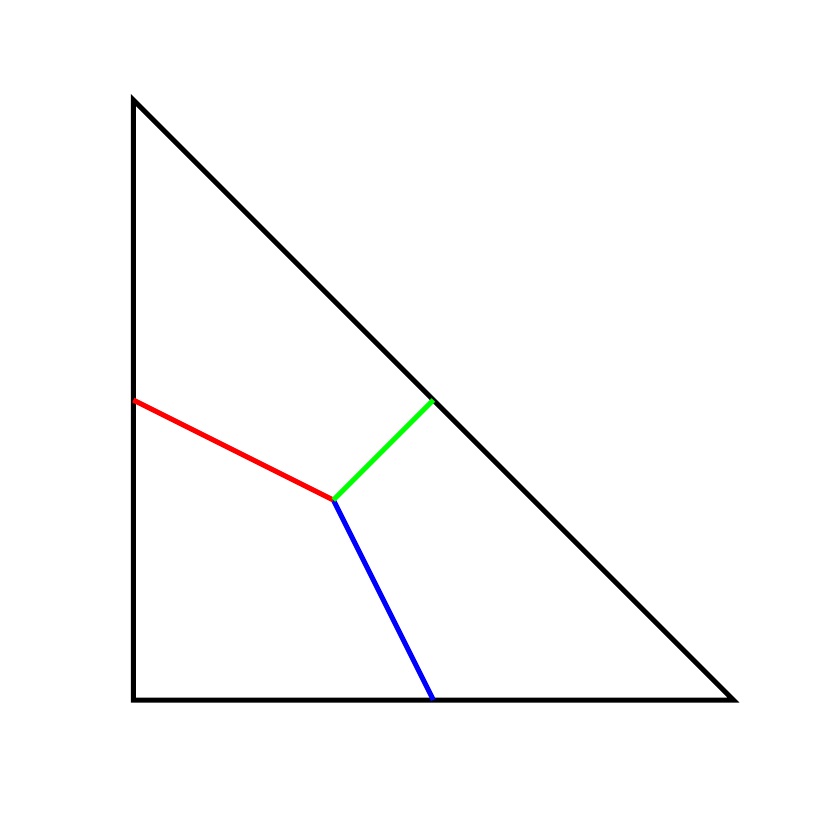}
\caption{The moment polytope of $\CP^2$, with the trisection decomposition described.}
\label{fig:toric}
\end{figure}

The toric geometry of $\CP^2$ yields a trisection $\cT$ as follows.  Define the moment map $\mu: \CP^2 \rightarrow \RR^2$ by the formula
\[\mu([z_1:z_2:z_3]) \coloneqq \left( \frac{3 |z_1|}{|z_1| + |z_2| + |z_3|}, \frac{3 |z_2|}{|z_1| + |z_2| + |z_3|} \right).\]
The image of $\mu$ is the convex hull of the points $\{ (0,0),(3,0),(0,3) \}$.  The fiber of $\mu$ over an interior point is $T^2$; the fibers over an interior point of a face of the polytope is $S^1$; and the fiber over a vertex is a point.  The preimage of an entire face of the polytope is a complex line $L_{\lambda} = \{[z_1:z_2:z_3] : z_{\lambda} = 0\}$ for some $\lambda$.  

The barycentric subdivision of the simplex $\mu(\CP^2)$ lifts to a trisection decomposition of $\CP^2$.  Define subsets
\begin{align*}
X_{\lambda} &\coloneqq \left\{ [z_1:z_2:z_3] : |z_{\lambda}|,|z_{\lambda+1}| \leq |z_{\lambda-1}| \right\}, \\
H_{\lambda} &\coloneqq \left\{ [z_1:z_2:z_3] : |z_{\lambda}| \leq  |z_{\lambda-1}| = |z_{\lambda+1}| \right\} .
\end{align*}
In the affine chart on $\CP^2$ obtained by setting $z_3 = 1$, the handlebody $X_1$ is exactly the polydisk 
\[ \Delta = \DD \times \DD = \{(z_1,z_2) : |z_1|,|z_2| \leq 1 \}.\]
Its boundary is the union of two solid tori
\[ H_{1} = S^1 \times \DD \qquad \text{and} \qquad H_{2} = \DD \times S^1. \]
The triple intersection $X_1 \cap X_2 \cap X_3$ is the torus
\[\Sigma \coloneqq \{[e^{i \theta_1}: e^{i \theta_2}: 1]: \theta_1,\theta_2 \in [0,2\pi] \}.\]
Furthermore, the intersection $B_{\lambda} = L_{\lambda} \cap H_{\lambda}$ is a core circle of the solid torus.

We therefore have shown the following proposition.

\begin{proposition}
The decomposition $\CP^2 = X_1 \cup X_2 \cup X_3$ is a $(1,0)$-trisection.
\end{proposition}

\subsection{Bridge trisections}
\label{sub:bridge-trisections}

Let $\{\tau_i\}$ be a collection of properly embedded arcs in a handlebody $H$.  An arc collection is {\it trivial} if they can be simultaneously isotoped { rel boundary} to lie in $\del H$.  If $\{\tau_i\}$ is trivial, then there exists a collection of disjoint disks $\Delta = \{\Delta_i\}$, embedded in $H$, such that $\del \Delta_i = \tau_i \cup s_i$ where $s_i$ is an arc in $\del \Sigma$.  We call each $\Delta_i$ a {\it bridge disk} and the arc $s_i$ the {\it shadow} of $\tau_i$.  A {\it bridge splitting} of a link $L$ is the 3-manifold $Y$ is a decomposition $(Y,L) = (H_1, \tau_1) \cup_{\Sigma} (H_2,\tau_2)$ where $H_1,H_2$ are handlebodies and the arc collections $\tau_1,\tau_2$ are trivial.  Finally, a collection $\cD = \{\cD_i\}$ of properly embedded disks in a 1-handlebody $X$ are {\it trivial} if they can be simultaneously isotoped { rel boundary} to lie in $\del X$.

\begin{definition}
\label{def:bridge-trisection}
A $(b;c_1,c_2,c_3)$-{\it bridge trisection} of a knotted surface $(X,\cK)$ is a decomposition $(X,\cK) = (X_1,\cD_1) \cup (X_2,\cD_2) \cup (X_3,\cD_3)$ such that
\begin{enumerate}
\item $X = X_1 \cup X_2 \cup X_3$ is a trisection of $X$,
\item each $\cD_{\lambda}$ is a collection of $c_{\lambda}$ trivial disks in $X_{\lambda}$, and
\item each tangle $\tau_{\lambda} = \cD_{\lambda - 1} \cap \cD_{\lambda}$, for $i \neq j$, is { a trivial tangle in $H_{\lambda}$ consisting of $b$ arcs.}
\end{enumerate}
If $(X,\cK)$ admits a bridge trisection, we say that $\cK$ is in {\it bridge position} { with respect to the trisection $X = X_1 \cup X_2 \cup X_3$}.
\end{definition}

Let $K_{\lambda} \subset Y_{\lambda}$ be the boundary of the trivial disk system $\cD_{\lambda}$.  Since $\cD_{\lambda}$ is trivial, the link $K_{\lambda}$ is the unlink with $c_{\lambda}$ components.  If $\cK$ is oriented, then each trivial disk system $\cD_{\lambda}$ inherits this orientation.  We choose to orient $\tau_{\lambda}$ by viewing it as a submanifold of $\del \cD_{\lambda}$.  With these conventions, the induced orientation on  the points of $\del \tau_{\lambda}$ is independent of $\lambda$ and moreover agrees with their induced orientation as the transverse intersection $\Sigma \pitchfork \cK$.  Finally, we have $(Y_{\lambda}, K_{\lambda}) = (H_{\lambda},\tau_{\lambda}) \cup_{\Sigma} (-H_{\lambda+1}, \tau^r_{\lambda+1})$ as oriented manifolds.

The main result of \cite{MZ-GBT} is that every knotted smooth surface $(X,\cK)$ can be put into bridge position.

\begin{theorem}[\cite{MZ-GBT}]
\label{thrm:MZ-GBT}
Let $\cT$ be a trisection of a closed, connected, oriented smooth 4-manifold $X$.  Every smoothly embedded surface $\cK$ in $X$ can be isotoped into bridge position with respect to $\cT$.
\end{theorem}

The {\it spine} of a surface $\cK$ in bridge position is the union $\tau_{1} \cup \tau_{2} \cup \tau_{3}$.  The spine uniquely determines the generalized bridge trisection of $\cK$ \cite[Corollary 2.4]{MZ-GBT}.  If $\cK$ admits a $(b;c_1,c_2,c_3)$ bridge trisection, then
\begin{equation}
\label{eq:Euler-GBT}
\chi(\cK) = c_1 + c_2 + c_3 - b.
\end{equation}

\subsection{Transverse bridge position}

In Section \ref{sec:transverse}, we introduce {\it transverse bridge position} and {\it transverse torus diagrams}.  The motivation was to find a class of bridge trisections and diagrams that have geometric rigidity, in analogy to grid diagrams for knots in $S^3$.  However, initial attempts suggest that it is unlikely for there to be a suitable notion of grid diagrams for surfaces in $\CP^2$.

In homogeneous coordinates, the handlebody $H_{\lambda}$ can equivalently be defined as 
\[ H_{\lambda} \coloneqq \left\{ [z_1:z_2:z_3] : |z_{\lambda}| \leq 1, |z_{\lambda+1}| = 1, z_{\lambda - 1} = 1 \right\}\]
Using standard polar coordinates
\[z_{\lambda} = r_{\lambda} e^{i \theta_{\lambda}} \qquad z_{\lambda+1} = r_{\lambda+1} e^{i \theta_{\lambda+1}}\]
we have coordinates $(\theta_{\lambda+1}, r_{\lambda},\theta_{\lambda})$ on $H_{\lambda} = S^1 \times \DD$.  The solid torus $H_{\lambda}$ is foliated by holomorphic disks.  The plane field tangent to this foliation is the kernel of the 1-form $d \theta_{\lambda+1}$.

The complex geometry of $\CP^2$ naturally induces contact structures on each 3-manifold $Y_{\lambda}$ of the trisection decomposition.  Specifically, each piece $X_{\lambda}$ of the trisection decomposition can be approximated by a Stein domain $\widehat{X}_{\lambda,N}$ in its interior and the hyperplane field $\widehat{\xi}_{\lambda,N}$ of complex tangencies on its boundary $\widehat{Y}_{\lambda,N} \cong S^3$ is the standard tight contact structure.  As $\widehat{X}_{\lambda,N}$ converges to $X_{\lambda}$, the contact structure $\widehat{\xi}_{\lambda,N}$ converges nonuniformly to the foliations of $H_{\lambda},-H_{\lambda+1}$ by holomorphic disks.   

A knotted surface $(\CP^2,\cK)$ in $\CP^2$ in general position with respect to the standard genus 1 trisection is {\it geometrically transverse} if, in each solid torus $H_{\lambda}$, the arcs of the spine are positively transverse to the foliation by holomorphic disks.  If $(\CP^2,\cK)$ is in bridge position and is geometrically transverse { (with a restricted model near the bridge points, see Section \ref{sub:trans-bridge})}, we say that it is in {\it transverse bridge position}.  If a surface is geometrically transverse, then for $N$ sufficiently large it intersects each $(\widehat{Y}_{N,\lambda},\widehat{\xi}_{\lambda,N})$ along a transverse link.  Furthermore, if it is in transverse bridge position then it intersects along transverse unlinks.

Every surface in transverse bridge position satisfies the adjunction inequality.  The degree, self-intersection number and Euler characteristic of $\cK$, along with the self-linking numbers of the transverse links in each $\widehat{Y}_i$, can be easily computed from a torus diagram.  Combining these with the Bennequin bound on the self-linking number yields the required bound.

\begin{theorem}
\label{thrm:transverse-Thom}
Let $(\CP^2,\cK)$ be a connected, oriented surface of degree $d$ in transverse bridge position.  Then $\cK$ satisfies the adjunction inequality:
\[ \chi(\cK) \leq 3d - d^2.\]
\end{theorem}

An immediate corollary is that there are surfaces in $\CP^2$ that cannot be put into transverse bridge position.  For example, nullhomologous spheres violate the adjunction inequality.  A natural question is therefore:

\begin{question}
Which surfaces can be isotoped into transverse bridge position?
\end{question}

It is unknown whether every essential surface in $\CP^2$ can be put into transverse bridge position.  All complex curves in $\CP^2$ can be isotoped into transverse bridge position \cite{LC-Meier}.  But the class of surfaces admitting transverse bridge presentations includes more than just complex curves and symplectic surfaces.  It is straightforward to attach handles and obtain surfaces in transverse bridge position with nonminimal genus, which therefore cannot be symplectic.

\subsection{Algebraic transversality and adjunction}

As mentioned above, there are surfaces in $\CP^2$ that cannot be isotoped into transverse bridge position.  To prove Theorem \ref{thrm:Thom} in full generality, we introduce the weaker notion of {\it algebraic transverse bridge position}.  

Recall that the solid torus $H_{\lambda} \cong S^1 \times \DD$ is foliated by holomorphic disks.  In polar coordinates on $H_{\lambda}$, the plane field tangent to the foliation is the kernel of the 1-form $d \theta_{\lambda+1}$.  A knotted surface $(\CP^2,\cK)$ in $\CP^2$ is {\it algebraically transverse} if for each $\lambda$, the integral of $d \theta_{\lambda+1}$ along each component of $\tau_{\lambda}$ is positive.  Clearly, a surface in transverse bridge position is also in algebraic transverse bridge position.  In addition, this geometric condition is sufficiently flexible to accomodate every surface of positive degree.

\begin{theorem}
Let $(\CP^2,\cK)$ be a connected, oriented surface of degree $d > 0$.  Then $\cK$ can be isotoped into algebraic transverse bridge position.
\end{theorem}

By further manipulations, we can completely isolate the obstruction to isotoping an algebraically transverse surface to be geometrically transverse.  Specifically, we can reduce to the case where the bridge trisection of a surface $\cK$ has a finite number of simple clasps (see Figure \ref{fig:simple-clasp}).  These clasps may be undone by a regular homotopy of the spine of the bridge trisection, which corresponds to a finger move of the surface $\cK$.  The result is an immersed but geometrically transverse surface that intersects each $\widehat{Y}_i$ along a transverse link.  Applying the ribbon-Bennequin inequality to a modification of this link, we can recover the adjunction inequality and prove Theorem \ref{thrm:Thom}.

\begin{theorem}
\label{thrm:alg-transverse-Thom}
Let $(\CP^2,\cK)$ be a connected, oriented surface of degree $d$ in algebraic transverse bridge position.  Then $\cK$ satisfies the adjunction inequality.
\[ \chi(\cK) \leq 3d - d^2.\]
\end{theorem}

\subsection{Acknowledgements}

I am deeply indebted to my long-term conversation partners, John Etnyre and Jeff Meier.  In addition, I would like to thank David Gay, Tye Lidman, Chuck Livingston, Gordana Matic, Paul Melvin and Alex Zupan for helpful comments and encouragement.  Finally, I would like to thank the referees for their careful reading and many suggestions.

%% file: diagram.tex
\section{Diagrams for surfaces in $\CP^2$}

Given a surface $(\CP^2,\cK)$ in general position with respect to the standard trisection, we can obtain a {\it torus diagram} $\cS(\cK) = (\cA,\cB,\cC)$ on the central surface $\Sigma$ of the trisection.  Algebraic information about $\cK$, including its homology class and normal Euler number, can be computed from the diagram $\cS(\cK)$.

\subsection{Trisection Diagram for $\CP^2$}
\label{sub:tri-diagram}

{ We can obtain a trisection diagram for the standard trisection of $\CP^2$ as follows.  Recall that $\Sigma = \{ [e^{i \theta_1}:e^{i \theta_2}:1] : \theta_1,\theta_2 \in [0,2\pi] \}$ is the central surface of the trisection.  In homogeneous coordinates, we have that
\[H_1 = \left\{ [z_1:e^{i \theta_2}:1] : |z_1| \leq 1, \theta_2 \in [0,2\pi]\right\}.\]
Thus the curve $\alpha \coloneqq \{ [e^{i \theta_1}: 1 : 1]\}$ bounds a disk in $H_1$.  Similarly, the curves $\beta \coloneqq \{ [1: e^{i \theta_2}: 1]\}$ and $\gamma \coloneqq \{ [1:1:e^{i \theta_3}] \}$ bound disks in $H_{2}$ and $H_{3}$, respectively.  Therefore, the triple $(\Sigma; \alpha,\beta,\gamma)$ is a trisection diagram for the standard trisection of $\CP^2$.} We will also use the notation
\[\alpha_1 \coloneqq \alpha \qquad \alpha_2 \coloneqq \beta \qquad \alpha_3 \coloneqq \gamma\]
when appropriate.

\subsection{Handlebody coordinates}

The natural coordinates on $\CP^2$ are homogeneous, not absolute.  Many of the arguments, definitions and statements are triply-symmetric and it is convenient to work in different affine charts on $\CP^2$.  We adopt the following convention.  When describing an object associated to a fixed but unspecificed $\lambda \in \{1,2,3\}$ --- such as $H_{\lambda},Y_{\lambda}, \tau_{\lambda}$, etc... ---- we will use the coordinates inherited from the affine chart $z_{\lambda - 1} = 1$.

For example, when $\lambda = 2$, we set $z_1 = 1$ in homogeneous coordinates and obtain affine coordinates $z_2,z_3$.  In polar form, we then have
\[z_2 = r_2 e^{i \theta_2} \qquad z_3 = r_3 e^{i \theta_3}\]
These restrict to give coordinates $(\theta_3,r_2,\theta_2)$ on $H_2 \cong S^1 \times \DD$.  However, the solid torus $H_2$ (with the opposite orientation) is also contained in $Y_1$.  It has a second coordinate system, denoted by $(\theta_1,r_2,\theta_2)$, that is induced by setting $z_3 = 1$.  Beware that despite equivalent notation, the angular coordinate $\theta_2$ differs between the two systems and depends on context (whether $\lambda = 1$ or $\lambda = 2$).

\subsection{Orientations}  

The standard orientation on $\CP^2$ orients each of the pieces of the trisection as follows.  Let $Y_{\lambda}$ be oriented as the boundary of $X_{\lambda}$, with outward-normal-first convention.  In particular, in the affine chart obtained by setting $z_{\lambda - 1} = 1$, we have coordinates
\[z_{\lambda} = r_{\lambda} e^{i \theta_{\lambda}} \qquad z_{\lambda+1} = r_{\lambda+1} e^{i \theta_{\lambda+1}}\]
and a frame $\{\del_{r_{\lambda}},\del_{\theta_{\lambda}},\del_{r_{\lambda+1}},\del_{\theta_{\lambda+1}}\}$ for $TX_{\lambda}$.  Along $H_{\lambda}$, the vector $\del_{r_{\lambda+1}}$ is the outward normal to $X_{\lambda}$, so the frame $\{\del_{\theta_{\lambda+1}},\del_{r_{\lambda}},\del_{\theta_{\lambda}}\}$ determines the orientation on $Y_{\lambda}$.  We fix an oriention $H_{\lambda} \subset Y_{\lambda}$ by restriction.  Finally, we orient the central surface $\Sigma$ as the boundary of $H_{\lambda} \subset Y_{\lambda}$, with its induced orientation.  Since $\del_{r_{\lambda}}$ is the outward normal, we get an oriented frame $\{\del_{\theta_{\lambda}},\del_{\theta_{\lambda+1}}\}$ on $\Sigma$.  As the construction is triply-symmetric, the induced orientation on the central surface $\Sigma$ is well-defined.

The canonical orientation of the holomorphic disks induces an orientation on each curve $\alpha,\beta,\gamma$.  In homology, we have that
\[ [\gamma] = - [\alpha] - [\beta].\]
Furthermore, each pair
\[\{[\alpha],[\beta]\} \qquad \{[\beta],[\gamma]\} \qquad \{ [\gamma],[\alpha]\}\]
is an oriented basis for $H_1(\Sigma ;\ZZ)$.  

\subsection{Surfaces in $\CP^2$}
\label{sub:surfaces-CP2}

Let $(\CP^2,\cK)$ be an immersed surface.  After a perturbation, we can assume that $\cK$ is in general position with respect to the genus-1 trisection of $\CP^2$.  Specifically, the surface $\cK$ intersects the central surface $\Sigma$ transversely in $2b$ points; that $\cK$ intersects each solid torus $H_{\lambda}$ transversely along a tangle $\tau_{\lambda}$; and that all of the self-intersections of $\cK$ are disjoint from the spine of the trisection.  By abuse of terminology, we refer to the points of $\cK \cap \Sigma$ as the {\it bridge points} of $\cK$ and $b$ as the {\it bridge index} of $\cK$.  Moreover, after a perturbation we can assume that each tangle $\tau_{\lambda}$ is disjoint from the core $B_{\lambda}$ of $H_{\lambda}$.  

Recall that we have chosen orientations on each handlebody $H_{\lambda}$ and the central surface $\Sigma$ is oriented.  If $\cK$ is oriented, we get orientations on the $2b$ points of $\cK \pitchfork \Sigma$ and the $b$ arcs of $\tau_{\lambda} = \cK \pitchfork H_{\lambda}$.  For a bridge point $v$, let $\sigma(v)$ denote this orientation.  Since $\Sigma$ is nullhomologous, the algebraic intersection number $[\cK] \cdot [\Sigma]$ vanishes and so we exactly $b$ positive bridge points and $b$ negative bridge points.  The orientation on a bridge point agrees with its orientation as the boundary of every tangle arc.  In particular, if the oriented boundary of some arc $\tau_{\lambda,i}$ is $v_1 - v_2$, then $\sigma(v_1) = 1$ and $\sigma(v_2) = -1$.  

\begin{figure}[h!]
\centering
\includegraphics[width=.5\textwidth]{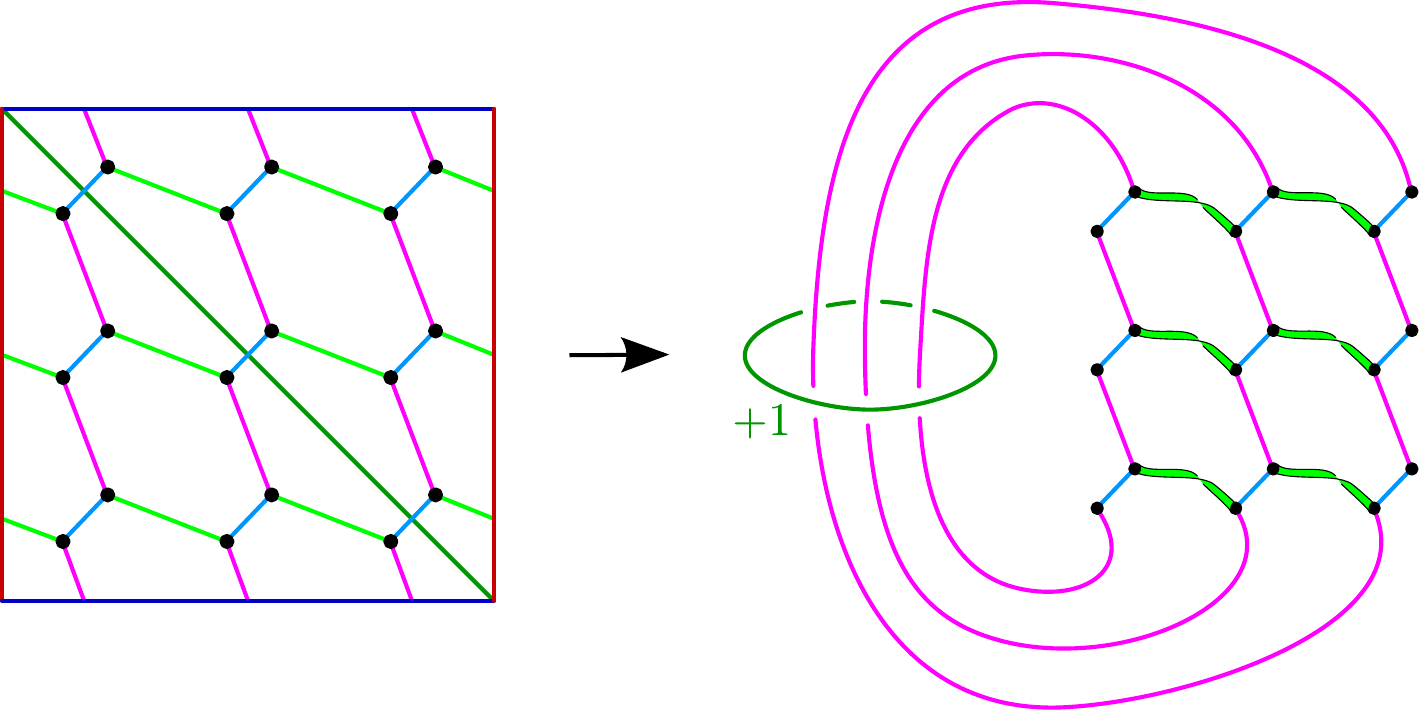}
\caption{({\it Left}) A torus diagram for a bridge trisection of a cubic curve in $\CP^2$.  ({\it Right}) A banded link diagram corresponding to the bridge splitting of the cubic.}
\label{fig:Banded_Link_Cubic}
\end{figure}

\subsection{Torus diagrams of surfaces in $\CP^2$}
\label{sub:torus-diagram}

Define the projection map $\pi_{\lambda}: H_{\lambda} \smallsetminus B_{\lambda} \longrightarrow \Sigma$ in coordinates by
\[\pi_{\lambda}(\theta_{\lambda+1},r_{\lambda},\theta_{\lambda}) \coloneqq (\theta_{\lambda}, \theta_{\lambda+1})\]
Let $(\CP^2,\cK)$ be an immersed surface in general position.  Set $\cA = \pi_1(\tau_{1})$, $\cB = \pi_2(\tau_{2})$ and $\cC = \pi_3(\tau_{3})$.  In addition, we will use the notation $\cA_{\lambda} \coloneqq \pi_{\lambda}(\tau_{\lambda})$.  After a perturbation of $\cK$, we can assume that the projections $\cA,\cB,\cC$ are mutually transverse and self-transverse, with intersections away from the bridge points.  In diagrams, our color conventions are that $\cA$ consists of red arcs, $\cB$ consists of blue arcs, and $\cC$ consists of green arcs.  A torus diagram for the cubic curve in $\CP^2$ is given in Figure \ref{fig:Banded_Link_Cubic}.

The orientation on the tangles induces orientations on the projections.  We can therefore interpret $\cA,\cB,\cC$ as oriented 1-chains on $T^2$ satisfying
\[\del \cA = \del \cB = \del \cC\]
The closed 1-chains
\[
\cS(K_1)  \coloneqq \cA - \cB \qquad \cS(K_2) \coloneqq \cB - \cC \qquad \cS(K_3) \coloneqq \cC - \cA
\]
are the projections of the oriented links $K_1,K_2,K_3$ onto the central surface.  These projections may be homologically essential in $\Sigma$, living in the classes
\begin{align*}
[\cS(K_1)] &= p_1 [\alpha] + q_1 [\beta] \\
[\cS(K_2)] &= p_2[\beta] + q_2 [\gamma] \\
[\cS(K_3)] &= p_3 [\gamma] + q_3 [\alpha] 
\end{align*}
for some integers $\{p_{\lambda},q_{\lambda}\}$. 

We define a secondary sign $\epsilon(v)$ for each bridge point $v$, according to the cyclic ordering of the incoming shadows.  If the three incoming arcs of $\cA,\cB,\cC$ at $v$ are positively cyclically ordered with respect to the orientation on $\Sigma$, we set $\epsilon(v) = 1$; otherwise we set $\epsilon(v) = -1$.  See Figure \ref{fig:signs}.

Finally, we will use the following convention to determine over-/undercrossings.  Recall we have a Heegaard decomposition $Y_{\lambda} = H_{\lambda} \cup - H_{\lambda+1}$.  We view $\Sigma$ from the perspective of the core of $-H_{\lambda+1}$, so that the tangle $\tau_{\lambda+1}$ is in the foreground and $\tau_{\lambda}$ is in the background.  

Thus, the arcs of $\pi_{\lambda+1}(\tau_{\lambda+1})$ always pass over the arcs of $\pi_{\lambda}(\tau_{\lambda})$.  In absolute terms, the arcs of $\cB$ always pass over the arcs of $\cA$, the arcs of $\cC$ always pass over the arcs of $\cB$, and the arcs of $\cA$ always pass over the arcs of $\cC$.  Using the color conventions of red, blue, and green for $\cA,\cB,\cC$, respectively, we have the convention: {\it blue over red, green over blue, red over green}.

For self-intersections of $\pi_{\lambda}(\tau_{\lambda})$, the strand further from $\Sigma$ passes under the strand closer to $\Sigma$.  The opposite occurs for self-intersections of $\pi_{\lambda+1}(\tau_{\lambda+1})$.  Thus, the crossing information for $\pi_{\lambda}(\tau_{\lambda})$ depends on whether the ambient manifold is $Y_{\lambda}$ or $Y_{\lambda - 1}$.  When drawing diagrams, as in Figure \ref{fig:crossing-example}, we always assume that ambient manifold is $Y_{\lambda}$.

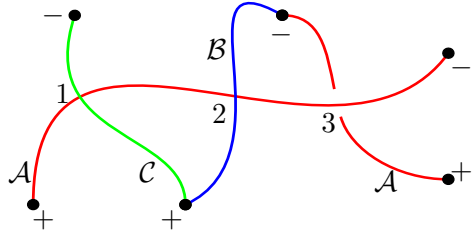
\begin{figure}[h!]
\centering
\labellist
	\large\hair 2pt
	\pinlabel $1$ at 60 143
	\pinlabel $+$ at 40 10
	\pinlabel $+$ at 175 10
	\pinlabel $-$ at 50 225
	\pinlabel $-$ at 290 205
	\pinlabel $2$ at 225 122
	\pinlabel $3$ at 340 112
	\pinlabel $+$ at 480 60
	\pinlabel $-$ at 480 170
	\pinlabel $\cA$ at 15 60
	\pinlabel $\cC$ at 150 60
	\pinlabel $\cB$ at 222 190
	\pinlabel $\cA$ at 400 50
\endlabellist
\includegraphics[width=.4\textwidth]{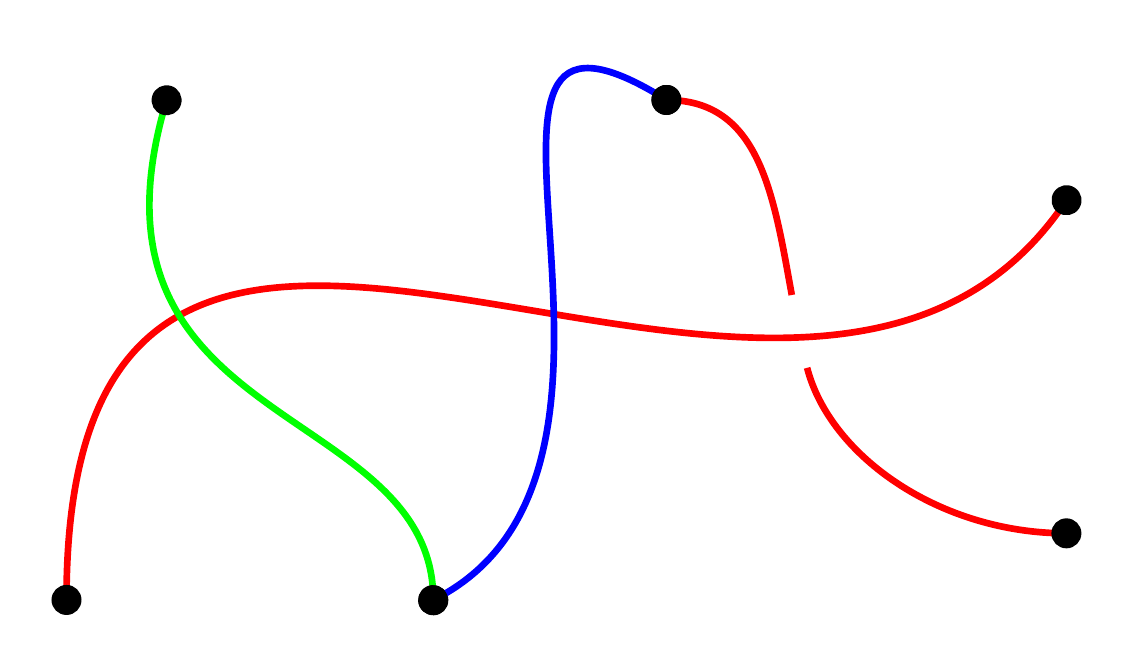}
\caption{An example piece torus diagram.  Crossing 1 contributes $-1$ to the writhe $w(K_3)$.  Crossing 2 contributes $+1$ to the writhe $w(K_1)$.  Crossing 3 contributes $+1$ to $w(K_1)$ and $-1$ to $w(K_3)$. }
\label{fig:crossing-example}
\end{figure}

\begin{figure}
\centering
\labellist
	\large\hair 2pt
	\pinlabel $\epsilon=1$ at 75 40
	\pinlabel $\epsilon=-1$ at 225 40
\endlabellist
\includegraphics[width=.9\textwidth]{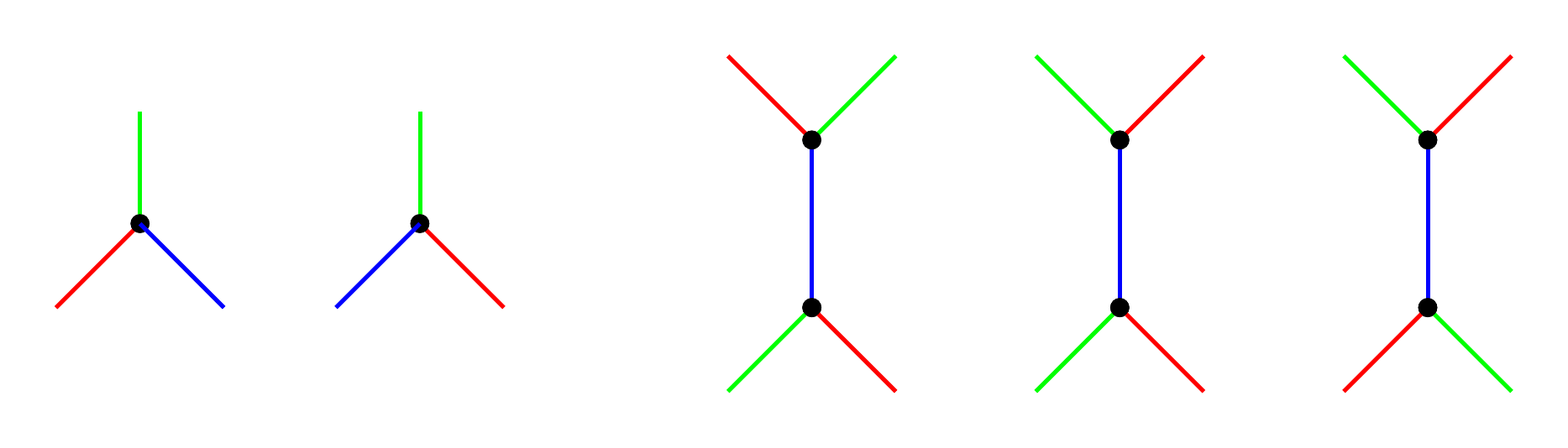}
\caption{{\it (Left)} A positive and negative bridge point in the shadow diagram.  {\it (Right)} A positively twisted, untwisted, and negatively twisted band in the shadow diagram.}
\label{fig:signs}
\end{figure}

\subsection{Degree formulas}

Let $L_{\lambda}$ denote the complex line $\{z_{\lambda} = 0\}$ in $\CP^2$.  This complex line intersects the handlebody $H_{\lambda}$ along a core $B_{\lambda}$ of the solid torus, geometrically dual to the compressing disk bounded by $\alpha_{\lambda}$.  We view $B_{\lambda}$ as an oriented knot in $H_{\lambda} \subset Y_{\lambda} = \del X_{\lambda}$, oriented as the boundary of the disk $\cL_{\lambda,N} \coloneqq L_{\lambda} \cap X_{\lambda}$.  The mirror image with the reverse orientation $-B^r_{\lambda}$, which is an oriented knot in $-H_{\lambda} \subset Y_{\lambda - 1} = \del X_{\lambda-1}$, is also the oriented boundary of the disk $\cL_{\lambda,S} \coloneqq L_{\lambda} \cap X_{\lambda - 1}$.

\begin{proposition}
\label{prop:degree-formula}
Let $(\CP^2,\cK)$ be an immersed, oriented surface in general position with respect to the standard trisection.  The degree $d$ of $\cK$ is given by the following formulas.
\begin{enumerate}
\item Let $L_{\lambda}$ be the complex line $\{z_{\lambda} = 0\}$.  Then
\[d = [L_{\lambda}] \cdot [\cK]\]
where $\cdot$ denotes the intersection pairing on $H_2(\CP^2;\ZZ)$.
\item Let $B_{\lambda} \subset H_{\lambda}$ denote the intersection of $L_{\lambda}$ with $H_{\lambda}$.  Then
\[d = lk_{Y_{\lambda}} (K_{\lambda}, B_{\lambda}) + lk_{Y_{\lambda - 1}} (K_{\lambda - 1},-B^r_{\lambda})\]
\item Let $[\alpha_{\lambda - 1}],[\alpha_{\lambda + 1}]$ denote classes in $H_1(T^2;\ZZ)$.  Then
\[d =\langle \cS(K_{\lambda}), [\alpha_{\lambda+1}] \rangle  + \langle [\alpha_{\lambda - 1}],\cS(K_{\lambda - 1}) \rangle\]
where $\langle , \rangle$ denotes the intersection pairing on $H_1(T^2;\ZZ)$.
\end{enumerate}
\end{proposition}

\begin{proof}
In $\CP^2$, the degree of $\cK$ is given by the algebraic intersection number of $\Sigma$ with any surface of degree 1.  We can choose this surface to be $L_{\lambda}$.

The complex line $L_{\lambda}$ decomposes as the union $\cL_{\lambda,N} \cup \cL_{\lambda,S}$.  It follows immediately that
\[ d = \sum_{x \in L_{\lambda} \pitchfork \cK} \sigma(x) =\sum_{y \in \cL_{\lambda,N} \pitchfork \cD_{\lambda}} \sigma(y) + \sum_{y \in L_{\lambda,S} \pitchfork \cD_{\lambda - 1}} \sigma(y)\]
where $\sigma$ denotes the sign of the intersection.  The surfaces $(\cD_{\lambda},K_{\lambda})$ and $(\cL_{\lambda,N},B_{\lambda})$ are properly { immersed} in $(X_{\lambda},Y_{\lambda})$ and the surfaces  $(\cD_{\lambda - 1},K_{\lambda - 1})$ and $(L_{\lambda,S},-B^r_{\lambda})$ are properly { immersed} in $(X_{\lambda - 1},Y_{\lambda - 1})$.  Consequently, 
\[\sum_{y \in \cL_{\lambda,N} \pitchfork \cD_{\lambda}} \sigma(y) = lk_{Y_{\lambda}}(K_{\lambda},B_{\lambda}) \qquad \sum_{y \in \cL_{\lambda,S} \pitchfork \cD_{\lambda - 1}} \sigma(y) = lk_{Y_{\lambda - 1}}( K_{\lambda - 1},-B^r_{\lambda})  \]
and the second formula follows.

Finally, we can compute linking numbers via the intersection pairing on $H_1(\Sigma;\ZZ)$.  Let $F$ be a compressing disk bounded by $\alpha_{\lambda + 1}$ in $-H_{\lambda+1}$ and extend it into $H_{\lambda}$ to obtain a Seifert surface $\widehat{F}$ for $B_{\lambda}$ in $Y_{\lambda}$.  By an isotopy, we can assume that there is a one-to-one correspondence between points of $K_{\lambda} \pitchfork \widehat{F}$ in $Y_{\lambda}$ and points of $\alpha_{\lambda+1} \pitchfork \cS(K_{\lambda})$ in $T^2$.  Counting with signs shows that
\[lk_{Y_{\lambda}} (K_{\lambda},B_{\lambda}) = \langle \cS(K_{\lambda}),[\alpha_{\lambda + 1}] \rangle .\]
A similar argument shows that
\[lk_{Y_{\lambda - 1}} (K_{\lambda -1},-B^r_{\lambda}) = \langle [\alpha_{\lambda - 1}], \cS(K_{\lambda - 1}) \rangle\]
and the third formula follows from the second.
\end{proof}

\begin{corollary}
\label{cor:deg-freedom}
Let $(\CP^2,\cK)$ be knotted surface of degree $d$ in bridge position with shadow diagram $\cS$.  Then there exist integers $p,q,r$ such that, in $H_1(T^2;\ZZ)$, the links represent the homology classes:
\begin{align*}
[\cS(K_{1})] &= p \cdot  [\alpha] + (d-q) \cdot [\beta] \\
[\cS(K_{2})] &= q \cdot [\beta] + (d-r) \cdot [\gamma] \\
[\cS(K_{3})] &= r  \cdot [\gamma] + (d-p) \cdot [\alpha]
\end{align*}
\end{corollary}

\begin{proof}
This follows immediately from Proposition \ref{prop:degree-formula} since the intersection pairing on $H_1(T^2)$ is nondegenerate.
\end{proof}

\subsection{Bridge stabilization}

There is a natural notion of stabilization for surfaces in bridge position that increases the bridge index by 1 \cite{MZ-GBT}.  The only type of stabilization we will need in this paper is what we will call a {\it mini stabilization}.  In dimension 4, a mini-stabilization corresponds to an isotopy given by a finger move of the surface $\cK$ through the central surface $\Sigma$ of the trisection.  Specifically, take a neighborhood in $\cK$ of an arc of $\tau_{\lambda}$ and push it towards the central surface.  After pushing through $\Sigma$, the isotoped surface $\cK$ now intersects $X_{\lambda + 1}$ in an extra trivial disk.  Diagrammatically, this can be seen as creating two bridge points along some arc of the shadow $\cA_{\lambda}$ and adding a new arc to both $\cA_{\lambda +1}$ and $\cA_{\lambda - 1}$.  See Figure \ref{fig:mini-stabilization}.  It is clear that this corresponds to a bridge stabilization of the links $K_{\lambda - 1}$ and $K_{\lambda}$, while introduces an extra unlinked, unknotted component to $K_{\lambda + 1}$.  Thus, we still have a bridge trisection of $\cK$.

\begin{figure}[h!]
\centering
\labellist
	\large\hair 2pt
	\pinlabel $\cA$ at 90 85
	\pinlabel $\cB$ at 350 115
	\pinlabel $-$ at 383 130
	\pinlabel $+$ at 325 25
	\pinlabel $\cC$ at 380 85
\endlabellist
\includegraphics[width=.5\textwidth]{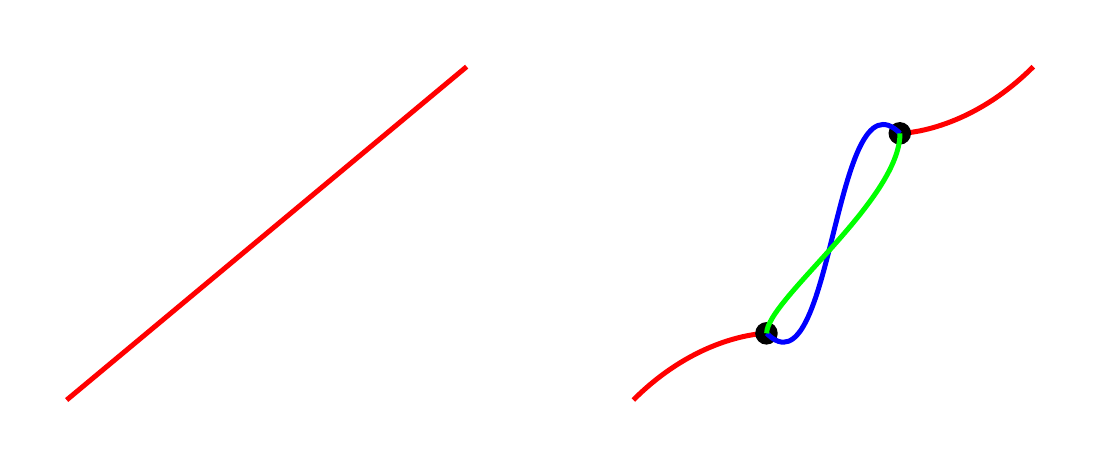}
\caption{A torus diagram depiction of a mini bridge stabilization. }
\label{fig:mini-stabilization}
\end{figure}
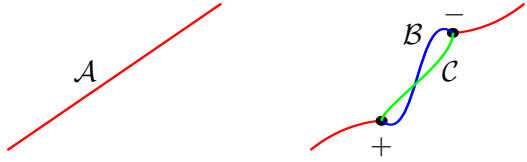

\subsection{Braid stabilization}

An isotopy of $\tau_{\lambda}$ that passes through $B_{\lambda}$ changes the homology classes represented by $\cS(K_{\lambda-1})$ and $\cS(\cK_{\lambda})$.  In particular, after the isotopy we obtain a new projection $\cA'_{\lambda}$ such that $[\cA_{\lambda} - \cA'_{\lambda}] = j[\alpha_{\lambda}]$ in $H_1(T^2;\ZZ)$ for some integer $j$.  {\color{red} When $j = \pm 1$,} we refer to this as a {\it braid stabilization}.  See Figure \ref{fig:transverse-stab}.

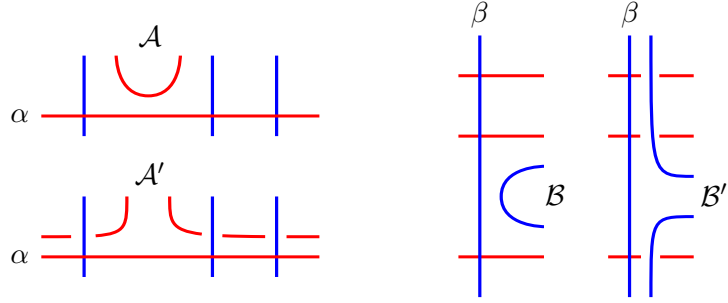
\begin{figure}[h!]
\centering
\labellist
	\large\hair 2pt
	\pinlabel $\alpha$ at 0 42
	\pinlabel $\cA'$ at 90 100
	\pinlabel $\alpha$ at 0 140
	\pinlabel $\cA$ at 90 195
	\pinlabel $\beta$ at 422 210
	\pinlabel $\beta$ at 318 210
	\pinlabel $\cB$ at 370 85
	\pinlabel $\cB'$ at 480 85
\endlabellist
\includegraphics[width=.6\textwidth]{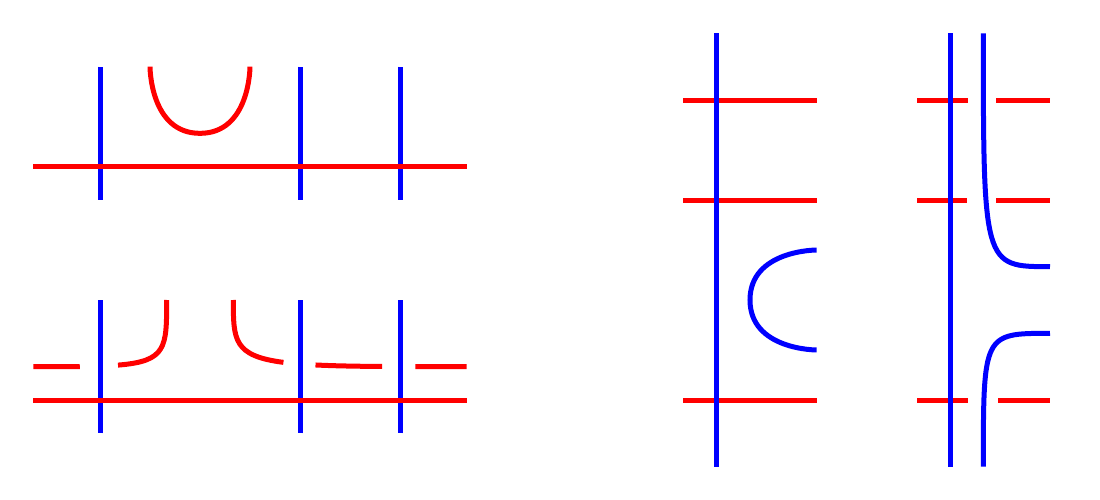}
\caption{{\it (Left:)} $\alpha$-stabilization.  {\it (Right:)} $\beta$-stabilization.}
\label{fig:transverse-stab}
\end{figure}

\begin{proposition}
Let $(\CP^2,\cK)$ be an immersed surface of degree $d$ in general position with torus diagram $\cS$.  Then for any integers $p,q,r$ there exists a sequence of braid stabilizations such that the links represent the following homology classes in $H_1(T^2;\ZZ)$:
\begin{align*}
\cS(K_{1}) &= p \cdot  [\alpha] + (d-q) \cdot [\beta] \\
\cS(K_{2}) &= q \cdot [\beta] + (d-r) \cdot [\gamma] \\
\cS(K_{3}) &= r  \cdot [\gamma] + (d- p) \cdot [\alpha]
\end{align*}
\end{proposition}

\begin{proof}
By Corollary \ref{cor:deg-freedom}, we can find some choice of integers $p',q',r'$.  Now perform $p-p'$ $\alpha$-stabilizations, $q - q'$ $\beta$-stabilizations, and $r - r'$ $\gamma$-stabilizations.
\end{proof}

\subsection{Surface framings}

In a torus diagram, we assign a formal writhe to each link projection $\cS(K_{\lambda})$ as follows.  Each $\cS(K_{\lambda}) = \cA_{\lambda} - \cA_{\lambda + 1}$ is a collection of oriented, self-transverse curves.  In addition, at each self-intersection point we have crossing information and can therefore assign a sign in the standard way.  Define the {\it writhe} { $w_{\lambda}(\cS(\cK))$} as the signed count of crossings of $\cS(K_{\lambda})$.  The writhe $w_{\lambda}(\cS(\cK))$ describes the { framing of $K_{\lambda}$ determined by pulling back the surface framing of $\cS(K_{\lambda})$}, up to a correction term determined by the homology class of $\cS(K_{\lambda})$.

\begin{lemma}
\label{lemma:surface-framing}
Suppose that $[\cS(K_{\lambda})] = p \cdot  [\alpha_{\lambda}] + q \cdot [\alpha_{\lambda+1}]$.  Then the framing on $K_{\lambda}$ induced by { the surface framing of }$\cS(K_{\lambda})$ is $w_{\lambda}(\cS(\cK)) + pq$.
\end{lemma}

\begin{proof}
If $\cS(K_{\lambda})$ lies in a disk on $\Sigma$, then $p = q = 0$ and the surface framing is exactly given by the writhe $w_{\lambda}(\cS(\cK))$.  Furthermore, any isotopy of $K_{\lambda}$ that induces a regular homotopy of $\cS(K_{\lambda})$ preserves the surface framing as well as $w_{\lambda}(\cS(\cK))$.  To obtain the formula, we just need to check that it does not change under braid stabilization.

Isotope $K_{\lambda}$ { to $K'_{\lambda}$} by a single braid stabilization through $B_{\lambda}$ and let $\cS(K'_{\lambda})$ denote the resulting projection.  Thus $[\cS(K'_{\lambda})] - [\cS(K_{\lambda})] = [\alpha_{\lambda}]$.  This does not change the surface framing but does change the signed count of crossings, as is evident in Figure \ref{fig:transverse-stab}, by
\[w_1(\cS(\cK')) = w_1(\cS(\cK)) - q.\]
Consequently
\[
w_1(\cS(\cK')) + (p+1)q = w_1(\cS(\cK)) - q + pq + q = w_1(\cS(\cK)) + pq.
\]
An identical argument shows that the sum $w_1(\cS(\cK)) + pq$ is invariant under braid stabilization passing through $-B^r_{\lambda+1}$ as well.
\end{proof}

\begin{proposition}
\label{prop:framing-conservation}
Let $(\CP^2,\cK)$ be an immersed surface of degree $d$ surface in general position with respect to the trisection and with torus diagram $\cS$.  Suppose that the projections of the three links of the bridge trisection represent the classes
\begin{align*}
[\cS(K_1)] &= p \cdot  [\alpha] + (d-q) \cdot [\beta] \\
[\cS(K_2)] &= q \cdot [\beta] + (d-r)  \cdot [\gamma] \\
[\cS(K_3)] &= r  \cdot [\gamma] + (d- p) \cdot [\alpha]
\end{align*}
in $H_1(\Sigma)$.  Then
\begin{align*}
\left\langle \cS(K_{\lambda}),\cS(K_{\lambda+1}) \right\rangle  & = d^2 -d(p + q + r) + pq + qr + rp \\
&= w_1(\cS(\cK)) + w_2(\cS(\cK)) + w_3(\cS(\cK)) + \frac{1}{2} \sum_{v} \epsilon_{\cS} (v) .
\end{align*}
\end{proposition}

\begin{proof}
The computation of the algebraic intersection number $\left\langle \cS(K_{\lambda}),\cS(K_{\lambda+1}) \right\rangle$ follows immediately since
\[ \langle [\alpha], [\beta] \rangle = 1 \qquad \text{and} \qquad [\gamma] = - [\alpha] - [\beta].\]
In addition, we can also compute the algebraic intersection number $\left\langle \cS(K_{\lambda}),\cS(K_{\lambda+1}) \right\rangle$ from the torus diagram $\cS$.  There are five types of potential contributions: (1) $\cA/\cB$ crossings, (2) $\cB/\cC$ crossings, (3) $\cC/\cA$ crossings, (4) self-intersections of $\cA,\cB,\cC$, and (5) $\cA_{\lambda+1}$-arcs.  Contributions may arise from $\cA_{\lambda+1}$ arcs because the projections $\cS(K_{\lambda})$ and $\cS(K_{\lambda+1})$ coincide along these arcs. 

The first three contribute to $w_1(\cS(\cK)),w_2(\cS(\cK)),w_3(\cS(\cK))$, respectively, and with our orientation and crossing conventions, the signs agree.  Each self-intersection point of some $\cA_{\mu}$ contributes opposite signs to $w_{\mu}(\cS(\cK))$ and $w_{\mu-1}(\cS(\cK))$, since the crossing data changes.  Thus the fourth contribute 0 on net.

{
\begin{figure}[h!]
\centering
\includegraphics[width=.6\textwidth]{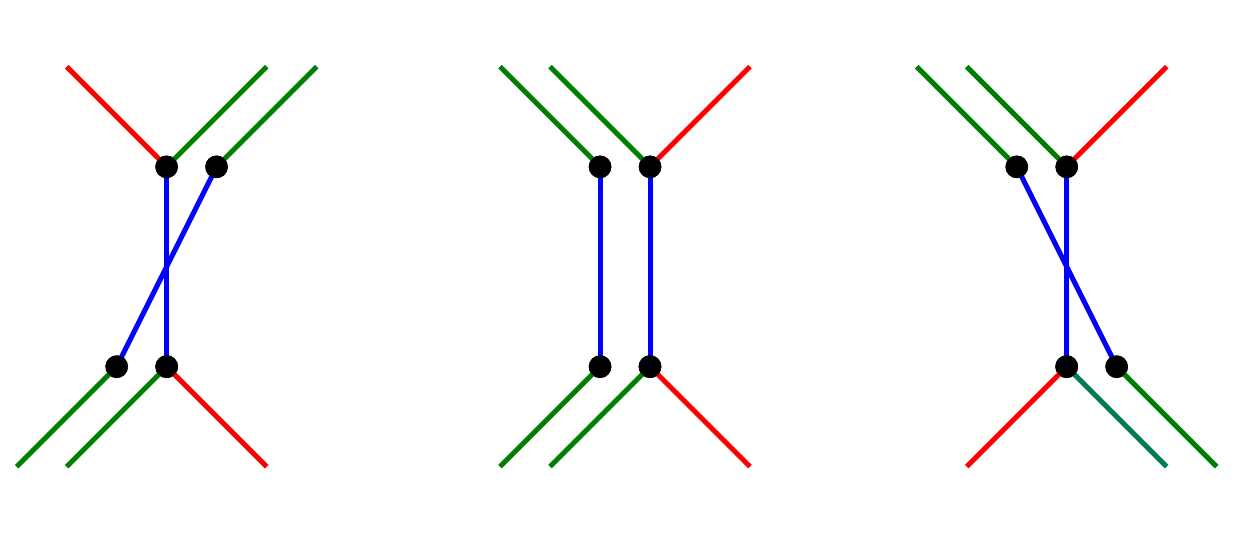}
\caption{Surface-framed pushoffs of the $\cB$-arcs in Figure \ref{fig:signs}}
\label{fig:signs-pushoff}
\end{figure}
}

Finally, the remaining contributions to the algebraic intersection number come from the $\cA_{\lambda+1}$ arcs.  { Take a surface-framed pushoff of the diagram $\cS$ to $\cS' = (\cA_1',\cA_2',\cA_3')$ such that $\cA'_{\lambda+1}$ is disjoint from $\cA_{\lambda+1}$.  See Figure \ref{fig:signs-pushoff}.}  Let $v_1,v_2$ be the endpoint of some arc $\cA_{\lambda,i}$.  If $\epsilon(v_1) = \epsilon(v_2) = 1$, then we can perturb $\cS(K_{\lambda+1})$ to add a single positive intersection point.  Similarly, if $\epsilon(v_1) = \epsilon(v_2) = -1$, a perturbation yields a negative intersection point.  Finally, if $\epsilon(v_1) = - \epsilon(v_2)$, we can perturb $\cS(K_{\lambda+1})$ and locally remove any intersection point along the arc.  The total contribution over all $\cA_{\lambda+1}$-arcs is exactly half the $\epsilon$-count of bridge points.
\end{proof}

When the homological data is triply symmetric and the bridge points are positively cyclically oriented, we get the following corollary, which relates the homological self-intersection number of $\cK$, the intersection pairing applied to the shadow $\cS$, the writhe of the shadow $\cS$, and the bridge index.

\begin{corollary}
Let $(\CP^2,\cK)$ be a knotted surface of degree-$d$ surface in bridge position with shadow diagram $\cS$.  Suppose that the shadows of the three links of the bridge trisection represent the classes
\begin{align*}
[\cS(K_1)] &= d \cdot  [\alpha] + 0 \cdot [\beta] \\
[\cS(K_2)] &= d \cdot [\beta] + 0  \cdot [\gamma] \\
[\cS(K_3)] &= d \cdot [\gamma] + 0\cdot [\alpha]
\end{align*}
in $H_1(\Sigma)$ and $\epsilon_{\cS}(v) = 1$ for all bridge points.  Then
\[ d^2 =  \langle \cS(K_{\lambda}), \cS(K_{\lambda+1}) \rangle = w_1(\cS) + w_2(\cS) + w_3(\cS) - b.\]
\end{corollary}

\begin{remark}
Lemma \ref{lemma:surface-framing} quantifies the difference between the surface framing of $\cS(L_1)$ and the nullhomologous framing in terms of the writhe.  Thus, Proposition \ref{prop:framing-conservation} can also be interpreted as a linear constraint on the surface framings of the three links comprising the spine of the surface.
\end{remark}

%% file: transverse.tex
\section{Transverse bridge position}
\label{sec:transverse}

The complex geometry of $\CP^2$ naturally induces contact structures on each 3-manifold $Y_{\lambda}$ of the trisection decomposition.  

\subsection{The contact structure $(S^3,\xi_{join})$}

We may view $S^3$ as the join $S^1 \ast S^1$ and from this perspective construct the standard tight contact structure on $S^3$. 

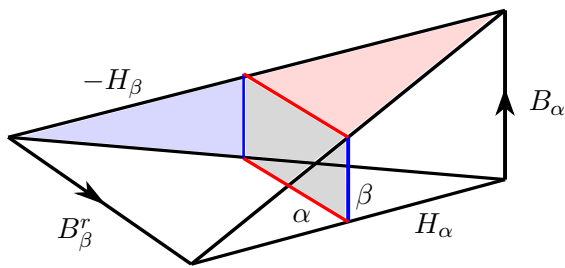
\begin{figure}[h!]
\centering
\labellist
	\large\hair 2pt
	\pinlabel $\beta$ at 203 70
	\pinlabel $\alpha$ at 170 60
	\pinlabel $-H_{\beta}$ at 70 130
	\pinlabel $H_{\alpha}$ at 240 55
	\pinlabel $B_{\alpha}$ at 300 120
	\pinlabel $B^r_{\beta}$ at 50 50
\endlabellist
\includegraphics[width=.5\textwidth]{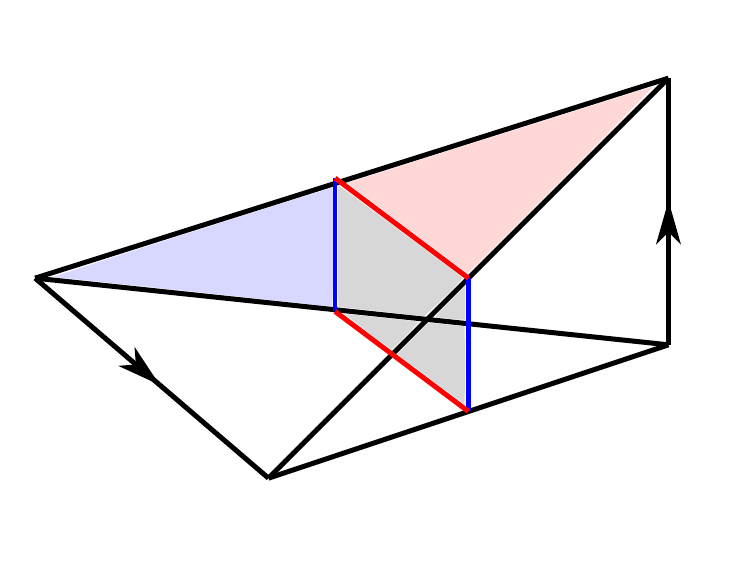}
\caption{$S^3$ viewed as the join $S^1 \ast S^1$, obtained by identifying the top and bottom faces and identifying the front and back faces.  The Heegaard surface $\left\{ \frac{1}{2} \right\} \times T^2$ is shaded in gray.  The $\beta$ curve bounds a blue compressing disk; the $\alpha$ curve bounds a red compressing disk.  The oriented edges become the positive Hopf link in $S^3$. }
\label{fig:S3join}
\end{figure}

Consider $M = [0,1] \times S^1 \times S^1$ with coordinates $(t,x,y)$.  Choose a smooth function $h(t): [0,1] \rightarrow [0,1]$ with $h(0) = 0$ and $h(1) = 1$ { such that $h' > 0$ { on $(0,1)$} and such that $h$ and $1 - h$ vanish to second order at $t = 0$ and $t =1$, respectively}.  { For the rest of the paper, we will let $h$ denote any such function}.  Define the contact structure
\[ \xi_h = \text{ker} \left( \sin \left(\frac{\pi}{2} h(t) \right) dx + \cos \left(\frac{\pi}{2} h(t) \right) dy \right).\]
At $t = 0$, the contact form is { $dy$}, and at $t = 1$, the contact form is {$dx$}, and for $t \in (0,1)$, the contact planes turn monotonically counter-clockwise through a total angle of $\frac{\pi}{2}$.  Up to isotopy, the contact structure $\xi_h$ is independent of the function $h$.  { To obtain $S^1 \ast S^1$, identify $(1,x_0,y) \sim (1,x_0,y')$ and $(0,x,y_0) \sim (0,x',y_0)$.  This contact structure descends to the smooth contact structure $\xi_{join}$ on $S^1 \ast S^1 = S^3$.}

{ 
\begin{lemma}
\label{lemma:join-standard}
The contact structure $(S^3,\xi_{join})$ is the standard tight contact structure.
\end{lemma}

\begin{proof}
View $S^3$ as the unit sphere in $\CC^2$.  The standard contact structure is the kernel of the 1-form $\alpha = x_1 dy_1 - y_1 dx_1 + x_2 dy_2 - y_2 dx_2$.  Consider the map $\psi: [0,1] \times S^1 \times S^1 \rightarrow \CC^2$, defined by setting
\[\psi(t,x,y) = (t e^{i x}, \sqrt{1 - t^2} e^{i y})\]
The pulled back contact form is $\psi^*(\alpha) = t^2 d x + (1 - t^2) d y$.  The resulting contact structure is $\xi_h$ for $h = \frac{2}{\pi}\arctan(t^2/(1 - t^2))$.
\end{proof}
}

\begin{remark}
\label{rem:beware}
The expert reader should beware:
\begin{enumerate}
\item The contact structure $(S^3,\xi_{join})$ admits an open book decomposition with the positive Hopf link as its binding.  However, the Heegaard decomposition obtained from this open book does not agree with the { obvious} Heegaard decomposition here along $\{\frac{1}{2}\} \times T^2$.  In particular, the Heegaard surface here is not the union of two pages of the open book decomposition.
\item In the standard convention for front diagrams with the contact form $\alpha = dz - y dx$, the vector $\del_y$ points into the { $xz$-plane}.  However, we will adopt the convention that the vector $\del_t$ points out of the { $xy$-plane}.  In particular, at $t =1$ the contact structure is vertical and is cooriented to the right (positive $x$ direction), while at $t = 0$ the contact structure is horizontal and is cooriented to the top (positive $y$ direction).  
\end{enumerate}
\end{remark}

\subsection{Projections of transverse links}

{ Let $L$ be a nullhomologous transverse link in a contact manifold $(Y,\xi)$.  For simplicity, we assume that $e(\xi) = 0$.  The {\it self-linking number} of $L$ is a $\ZZ$-valued invariant, well-defined up to isotopy through transverse links.   Let $F$ be a connected, oriented Seifert surface for $L$. The bundle $\xi|_{F}$ is trivial and so admits a nonvanishing section $s$.  Along $L$, the section $s$ determines a framing and a pushoff $L'$ in the direction of this framing.  The self-linking  number is defined to be
\[sl(L,F) \coloneqq lk(L,L').\]
Since the Euler class of the contact structure $\xi$ vanishes, the self-linking number is independent of the surface $F$.}

Given a link $L \subset S^1 \ast S^1$ disjoint from the Hopf link, we can lift $L$ to a link in $M = [0,1] \times T^2$ and consider its projection onto the Heegaard surface $\Sigma = \{\frac{1}{2}\} \times T^2$.  By abuse of notation, we also refer to this link as $L$.  In $T^2$, set $\alpha = S^1 \times \{0\}$ and $\beta = \{0\} \times S^1$.

\begin{proposition}
\label{prop:torus-sl}
Let $L$ be a transverse link in $(S^3, \xi_{join})$ disjoint from the positive Hopf link $B_{\alpha} \cup - B^r_{\beta}$.  Lift $L$ to a link in $M = [0,1] \times T^2$ and project onto $\{\frac{1}{2}\} \times T^2$.  Let $\cD$ be the resulting diagram, let $w(\cD)$ denote the signed count of positive and negative crossings in $\cD$, and let $p [\alpha] + q [\beta]$ be the homology class of $L$.

The self-linking number of $L$ is given by the formula
\[sl(L) = w(\cD) + pq - p - q.\]
\end{proposition}

\begin{proof}
Choose a Seifert surface $F \subset S^3$ for $L$.  After multiplying by a suitable nonnegative function, the vector field $\del_t$ on $M$ descends to a section { $v$} of $\xi_M$ that vanishes only along the Hopf link $H = B_{\alpha} \cup -B^r_{\beta}$.  We can use this vector field to compute the self-linking number, provided we add a correction term corresponding to the intersection points $F \pitchfork B_{\alpha} \cup -B^r_{\beta}$.  { In particular, we obtain a section $v$ of $\xi_i|_{F}$ that vanishes precisely at the points $F \pitchfork H$.  Near each intersection point, we can choose Darboux coordinates $(x,y,z)$ with contact form $dz + xdy - y dz$, such that $F$ lies in the plane $\{z = 0\}$, that $H$ is the $z$-axis, and that $v = x \del_x + y \del_y$.  The sign of the zero of $v$, viewed as a section of $\xi_i|_{F}$, is the same as the sign of the intersection of $F$ with $H$.  Thus, the framings of $L$ determined by $v$ and by a nonvanishing section $s$ differ by $\sum_{x \in F \pitchfork H} \sigma(x)$, where $\sigma$ denotes the sign of the intersection.}

Now, let $L_t$ denote { a pushoff of $L$ by $\del_t$}.  Then we have
\begin{align*}
sl(L,F) &= lk(L,L_t) - \sum_{x \in F \pitchfork H} \sigma(x) \\
&= lk(L,L_t) - (lk(L,B_{\alpha}) + lk(L,-B^r_{\beta})) \\
&= lk(L,L_t) - p - q.
\end{align*}

The formula now follows by applying Lemma \ref{lemma:surface-framing} since $\del_t$ is transverse to $\left\{ \frac{1}{2} \right\} \times T^2$ and therefore determines the surface framing.  
\end{proof}

\begin{remark}
Similar formulas for the classical invariants of Legendrian and transverse knots were obtained in \cite{Dymara,Kawamuro-Pavelescu}.
\end{remark}

\subsection{The contact structures $(\widehat{Y}_{\lambda,N},\widehat{\xi}_{\lambda,N})$}

Recall that in the affine chart obtained by setting $z_{\lambda-1} = 1$, the sector $X_{\lambda}$ is exactly the polydisk $\Delta = \{ |z_{\lambda}|,|z_{\lambda+1}| \leq 1\}$.  This polydisk can be approximated by a holomorphically convex 4-ball.  Specifically, consider the function
\[f_{\lambda,N}(z_{\lambda},z_{\lambda+1}) \coloneqq \frac{1}{N}(|z_{\lambda}|^2 + |z_{\lambda+1}|^2) + |z_{\lambda}|^{2N} + |z_{\lambda+1}|^{2N}\]
for some $N \gg 0$. 

\begin{lemma}
\label{lemma:tri-stein}
Let $\widehat{X}_{\lambda,N}$ be the compact sublevel set $f_{\lambda,N}^{-1}((-\infty,1])$ of $f_{\lambda,N}$ for $N \gg 0$.
\begin{enumerate}
\item For $N$ sufficiently large, the level set $\widehat{Y}_{\lambda,N} \coloneqq \del \widehat{X}_{\lambda,N} = S^3$ is $C^0$-close to $Y_{\lambda}$.
\item Let $U$ be a fixed open neighborhood of the central surface $\Sigma$.  For $N$ sufficiently large, the level set $\widehat{Y}_{\lambda,N}$ is $C^{\infty}$-close to $Y_{\lambda}$ outside $U$.
\item $\widehat{X}_{\lambda,N}$ is a Stein domain contained in the interior of $X_\lambda$.
\item The field $\widehat{\xi}_{\lambda,N}$ of complex tangencies to $\widehat{Y}_{\lambda,N}$ is the standard tight contact structure.
\end{enumerate}
\end{lemma}

\begin{proof}
{ On any compact subset,} for large $N$, the function $\left(f_{\lambda,N}\right)^{\frac{1}{2N}}$ is a perturbation of the $L^{2N}$-norm.  Thus, the level set $f_{\lambda,N}^{-1}(1)$ is a perturbation of the unit circle in $\CC^2$ with respect to the $L^{2N}$-norm.  These level sets then must converge to $Y_{\lambda}$, which is the unit circle with respect to the $L^{\infty}$-norm.  This convergence is uniform outside a fixed open neighborhood of $\Sigma$.  { The complex Hessian is 

\[ - d d^{\CC} f_{\lambda,N} = \del \overline{\del} f_{\lambda,N} = \left( \frac{1}{N} + 2N|z_{\lambda}|^{2N-2} \right) dz_{\lambda} \wedge d \overline{z}_{\lambda} + \left( \frac{1}{N} + 2N|z_{\lambda+1}|^{2N-2} \right) dz_{\lambda+1} \wedge d \overline{z}_{\lambda+1}\]
This is clearly positive-definite and} the function $f_{\lambda,N}$ is strictly plurisubharmonic for all $N \geq 1$.  Therefore its sublevel sets are Stein domains { and} the field of complex tangencies along its boundary is a contact structure.  

{ Consider the family $f_{\lambda,N,t} = \frac{1}{N}(|z_{\lambda}|^2 + |z_{\lambda+1}|^2) + t\left( |z_{\lambda}|^{2N} + |z_{\lambda+1}|^{2N} \right)$.  For all $t \in [0,1]$, the function $f_{\lambda,N,t}$ is strictly plurisubharmonic with a single critical point at the origin.  Thus, as $t$ varies in $[0,1]$, the family $f^{-1}_{\lambda,N,t}(1)$ is a smooth isotopy of contact hypersurfaces, with $f^{-1}_{\lambda,N,0}(1)$ the sphere of radius $\sqrt{N}$ and $f^{-1}_{\lambda,N,1}(1) = \widehat{Y}_{\lambda,N}$}
\end{proof}

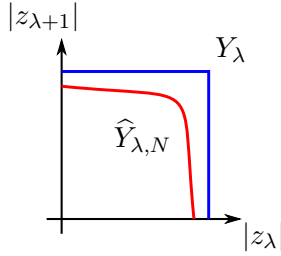
\begin{figure}[h]
\centering
\labellist
	\large\hair 2pt
	\pinlabel $|z_{\lambda}|$ at 210 10
	\pinlabel $|z_{\lambda+1}|$ at 10 210
	\pinlabel $Y_{\lambda}$ at 180 180
	\pinlabel $\widehat{Y}_{\lambda,N}$ at 100 100
\endlabellist
\includegraphics[width=.2\textwidth]{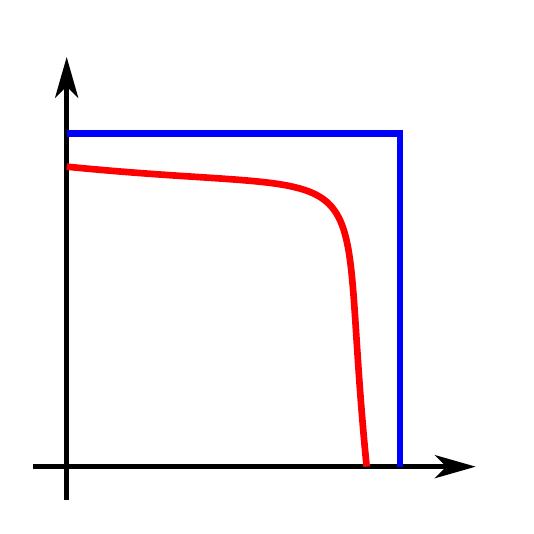}
\caption{As $N \rightarrow \infty$, the level set $\widehat{Y}_{\lambda,N}$ approaches $Y_{\lambda}$.}
\label{fig:ln-unit-circle}
\end{figure}

\subsection{Transverse bridge position}
\label{sub:trans-bridge}

The 4-dimensional picture above motivates the introduction of transverse bridge position and transverse shadow diagrams in this section.  Roughly speaking, a surface $(\CP^2,\cK)$ is in transverse bridge position if it is in bridge position with respect to the standard genus 1 trisection and it intersects each $(\widehat{Y}_{\lambda},\widehat{\xi}_{\lambda})$ in a transverse unlink. 

{ For unit complex numbers $x,y$, the complex line $x z_1 + y z_2 + z_3 = 0$ intersects $\Sigma$ at the points $[\overline{x} \zeta : \overline{y} \zeta^2 :1]$ and $[\overline{x}\zeta^2: \overline{y} \zeta:1]$, where $\zeta$ is a primitive $3^{\text{rd}}$-root of unity.  One of these intersections is positive and the other is negative.  We say that a surface $(\CP^2,\cK)$ has {\it complex bridge points} if for each point $[x:y:1] \in \cK \pitchfork \Sigma$, it locally agrees with this projective line through that point.  In particular, the surface $\cK$ locally agrees with the projective line $\left\{\frac{\zeta}{x} z_1 + \frac{\zeta^2}{y} z_2 + z_3 = 0\right\}$.}

{
\begin{lemma}
Suppose that $v \in \cK \pitchfork \Sigma$ is a complex bridge point.
\begin{enumerate}
\item Suppose that $\cK$ locally agrees with the projective line $\left\{\frac{\zeta}{x} z_1 + \frac{\zeta^2}{y} z_2 + z_3 = 0\right\}$ at $[x:y:1]$.  The sign of $v$ as an intersection point  of $\cK$ and $\Sigma$ is positive (resp. negative) if $\zeta = e^{2\pi i/3}$ (resp. if $\zeta = e^{-2\pi i/3}$.
\item We have $\epsilon_{\cS}(v) = 1$, where $\cS$ is the torus diagram of $\cK$.
\end{enumerate}
\end{lemma}

\begin{proof}
For both parts, it suffices to check at the point $[1:1:1]$, as every other point can be obtained by rotating this computation by the complex angle $(x,y)$.   

First, let $\zeta = -\frac{1}{2} + \sigma \frac{\sqrt{3}}{2} i$, with $\sigma \in \{\pm 1\}$.   In affine coordinates, the complex line is $\zeta z_1 + \zeta^2z_2 + 1 = 0$ and the vector $\zeta^2 \del_{z_1} - \zeta \del_{z_2}$ is tangent to this line.  Changing to rectangular coordinates, a basis for $T_{[1:1:1]}\cK$ is given by
\[u = - \del_{x_1} + \sigma \sqrt{3} \del _{y_1} + \del_{x_2} + \sigma \sqrt{3} \del_{y_2} \qquad Ju = -\sigma \sqrt{3} \del_{x_1} - \del_{y_1} - \sigma \sqrt{3} \del_{x_2} + \del_{y_2}\]
A basis for $T_{[1:1:1]}\Sigma$ is $\{\del_{y_1}, \del_{y_2}\}$.  The determinant of the tuple $\{u,Ju,\del_{y_1},\del_{y_2}\}$ is $\sigma \sqrt{3}$.

Second, the tangent plane $T_{[1:1:1]}H_1$ is spanned by the basis $\{ \del_{x_1},\del_{y_1},\del_{y_2}\}$.  The tangent vector to $\cK \cap H_1$ at $[1:1:1]$ is
\[\sigma \sqrt{3} u + J u = -\sigma 2 \sqrt{3} \del_{x_1} + 2 \del_{y_1} + 4 \del_{y_2}\]
Projecting onto $\Sigma$, the tangent vector of the arc $\cA$ at $v$ is $2 \del_{y_1} + 4 \del_{y_2}$.  The triple $\{\del_{y_1},2 \del_{y_1} + 4 \del_{y_2}, \del_{y_2}\}$ is positively cyclically ordered.  Similar calculations hold by cyclic symmetry.  In particular, the tangent vector to $\cK \cap H_2$ must lie between $\del_{y_2}$ and $- \del_{y_2} - \del_{y_1}$ and the tangent vector to $\cK \cap H_3$ must lie between $-\del_{y_1} - \del_{y_2}$ and $\del_{y_1}$.  Combining this information, we see that $\epsilon_{\cS}(v) = 1$.
\end{proof}
}

Given the above observation about the limiting behavior of $\widehat{\xi}_{\lambda}$, we make the following definition.

\begin{definition}
A knotted surface $(\CP^2,\cK)$ is {\it geometrically transverse} if
\begin{enumerate}
\item $\cK$ is in general position with respect to the standard trisection,
\item $\cK$ has complex bridge points, and
\item each tangle $\tau_{\lambda} = \cK \pitchfork H_{\lambda}$ is positively transverse to the foliation of $H_{\lambda}$ by holomorphic disks.
\end{enumerate} 
Furthermore, if $(\CP^2,\cK)$ is geometrically transverse and in bridge position, we say that it is in {\it transverse bridge position}. \end{definition}

\begin{proposition}
\label{prop:trans-implies-trans}
Let $(\CP^2,\cK)$ be in general position with respect to the standard trisection
\begin{enumerate}
\item If $\cK$ is geometrically transverse, then for $N$ sufficiently large, the intersection $\Khat_{\lambda} \coloneqq \cK \pitchfork \widehat{Y}_{\lambda,N}$ is a transverse link.
\item If $\cK$ is in transverse bridge position, then for $N$ sufficiently large, the intersection $\Khat_{\lambda} \coloneqq \cK \pitchfork \widehat{Y}_{\lambda,N}$ is a transverse unlink.
\end{enumerate}
\end{proposition}

\begin{proof}
By assumption, the surface $\cK$ has complex bridge points.  Thus, we can choose some $\epsilon > 0$ such that within the open set
\[U_{\epsilon} \coloneqq \left\{ \left| 1 - |z_1| \right| < \epsilon, \left| 1 - |z_2| \right| < \epsilon \right\}\]
the surface $\cK$ agrees with a collection of complex lines.  Furthermore, by Lemma \ref{lemma:tri-stein}, we can choose $N$ such that outside of $U_{\epsilon}$ the level set $\widehat{Y}_{\lambda,N}$ is $C^{\infty}$-close to $Y_{\lambda}$.  

The surface $\cK$ is geometrically transverse, hence it is transverse to the foliation of $Y_{\lambda} \smallsetminus U_{\epsilon}$ by holomorphic disks.  Since $\widehat{Y}_{\lambda,N}$ is $C^{\infty}$-close to $Y_{\lambda}$ outside $U_{\epsilon}$, we can assume that $\cK$ remains transverse to the hypersurface $\widehat{Y}_{\lambda,N}$ and positively transverse to its field of complex tangencies on $\widehat{Y}_{\lambda,N}$ outside $U_{\epsilon}$.  In other words, outside $U_{\epsilon}$, the intersection $\Khat_{\lambda}$ is a 1-manifold that is positively transverse to the contact structure $\widehat{\xi}_{\lambda,N}$ --- i.e. a tranverse link.

To complete the proof, we neet to check that these properties also hold within $U_{\epsilon}$. { The contact structure is the field of complex tangencies along $\widehat{Y}_{\lambda,N}$.  This complex line field is the kernel of the holomorphic 1-form $\del f_{\lambda,N}$ (see \cite[Lemma 5.3.3]{Geiges}).  We have that
\begin{equation}
\label{eq:del-f}
 \del f_{\lambda,N} = \left( \frac{1}{N}  + N|z_{\lambda}|^{2N-2}\right) \overline{z}_{\lambda} d z_{\lambda} + \left( \frac{1}{N}  + N|z_{\lambda+1}|^{2N-2}\right) \overline{z}_{\lambda+1} d z_{\lambda+1}
\end{equation}
The vector $x\zeta  \del_{z_1} - y \del_{z_2}$ spans the complex tangent line to $\cK$ at $[x:y:1]$, where $\zeta$ is a primitive $3^{\text{rd}}$-root of unity.  We then have
\[\del f_{\lambda,N}(\zeta y \del_{z_1} - x \del_{z_2}) = \left( \frac{1}{N} + N \right) \left( \overline{x} (\zeta x) + \overline{y}(-y) \right) = \left( \frac{1}{N} + N \right)(\zeta - 1) \neq 0\]
}
Part (2) follows immediately from Part (1), since if $\cK$ is in bridge position, it intersects each $Y_{\lambda}$, and therefore $\widehat{Y}_{\lambda,N}$, along an unlink.
\end{proof}

\subsection{Total self-linking number}

When $(\CP^2,\cK)$ is geometrically transverse, the total self-linking numbers of the three links $\Khat_1,\Khat_2,\Khat_3$ can be computed from a torus diagram  of $\cK$ and is completely determined by algebraic information of $\cK$.

{
\begin{lemma}
\label{lemma:torus-hat-sl}
Suppose that $\cK$ is geometrically transverse with torus diagram $\cS$.  If
\[[S(K_{\lambda})] = p_{\lambda} [\alpha_{\lambda}] + q_{\lambda} [\alpha_{\lambda+1}] \]
in homology, then
\[sl(\widehat{K}_{\lambda}) = w_{\lambda}(\cS(\cK)) + p_{\lambda} q_{\lambda} - p_{\lambda} - q_{\lambda}\]
\end{lemma}

\begin{proof}
This formula is the combination of Lemma \ref{lemma:surface-framing} with Proposition \ref{prop:torus-sl}.  In particular, we have a commutative diagram
\[ \xymatrix{
(0,1) \times T^2 \ar[rr]^{\Psi_{\lambda,N}} \ar[d] && \widehat{Y}_{\lambda,N} \smallsetminus (\widehat{\SH}_{\lambda,N}) \ar[r]^(.37){R_{\lambda}} \ar[d]^{\widehat{\pi}_{\lambda,N}} & (H_{\lambda} \smallsetminus B_{\lambda}) \cup (- H_{\lambda+1} \smallsetminus B_{\lambda + 1})  \ar[d]^{\pi_{\lambda} \cup \pi_{\lambda+1}} \\
\left\{\frac{1}{2} \right\} \times T^2 \ar[rr]^{\Psi_{\lambda,N}}&& \widehat{\Sigma}_{\lambda,N} \ar[r]^{R_{\lambda}} & \Sigma 
}\]
where
\begin{enumerate}
\item $\widehat{\SH}_{\lambda,N}$ is the Hopf link obtained as the intersection of $\widehat{Y}_{\lambda,N}$ with the lines $\{z_{\lambda} = 0\}$ and $\{z_{\lambda+1} = 0\}$.
\item $\widehat{\Sigma}_{\lambda,N}$ is the torus obtained as the intersection of $\widehat{Y}_{\lambda,N}$ with the subset $\{ (z_{\lambda},z_{\lambda+1}) : |z_{\lambda}| = |z_{\lambda+1}|\}$,
\item $\Psi_{\lambda,N}$ is the map
\[\Psi_{\lambda,N}(t,\theta_1,\theta_2) \coloneqq (r_{1,\lambda,N}(t) e^{i \theta_1}, r_{2,\lambda,N}(t) e^{i \theta_2}) \]
\item $R_{\lambda}$ is the radial projection map
\[\R_{\lambda}(r_1 e^{i \theta_1}, r_2 e^{i \theta_2}) \coloneqq 
\begin{cases}
\left(e^{i \theta_1}, \frac{r_2}{r_1} e^{i \theta_2} \right) & \text{if } r_1 \geq r_2 \\
\left(\frac{r_1}{r_2} e^{i \theta_1}, e^{i \theta_2} \right) & \text{if } r_1 \leq r_2
\end{cases}\]
\end{enumerate}
In order to use the self-linking formula of Proposition \ref{prop:torus-sl}, we need to check that $\Psi^*_{\lambda,N}(\widehat{\xi}_{\lambda,N}) = \xi_{h}$ for some function $h$.  Since $f_{\lambda,N}$ is strictly plurisubharmonic, the form $-d^{\CC}f_{\lambda,N} = -df_{\lambda,N} \circ J = 2 \text{Im}(\del f_{\lambda,N})$ is a contact form for the field of complex tangencies along $\widehat{Y}_{\lambda,N}$ \cite[Section 5.4]{Geiges}.  From Equation \ref{eq:del-f}, we have that 
\[2 \text{Im}(\del f_{\lambda,N}) = \left( \frac{1}{N}  + N|z_{\lambda}|^{2N-2}\right)( x_{\lambda} dy_{\lambda}- y_{\lambda} dx_{\lambda}) + \left( \frac{1}{N}  + N|z_{\lambda+1}|^{2N-2}\right) (x_{\lambda+1} dy_{\lambda+1} - y_{\lambda+1} dx_{\lambda+1}),\]
which in polar coordinates is
\[2 \text{Im}(\del f_{\lambda,N}) = \left( \frac{1}{N}  + Nr_{\lambda}^{2N-2}\right)r_{\lambda}^2 d \theta_{\lambda} + \left( \frac{1}{N}  + Nr_{\lambda+1}^{2N-2}\right) r_{\lambda+1}^2 d \theta_{\lambda+1}.\]
{ We can view $r_{\lambda}$ and $r_{\lambda+1}$ as functions of $t \in [0,1]$.}  Consequently, the pulled back contact structure is $\xi_h$ where
\[h(t) = \frac{2}{\pi} \arctan \left( \frac{\left( \frac{1}{N}  + Nr_{\lambda}^{2N-2}\right)r_{\lambda}^2}{ \left( \frac{1}{N}  + Nr_{\lambda+1}^{2N-2}\right) r_{\lambda+1}^2} \right)\]
\end{proof}

\begin{proposition}
\label{prop:sl-count}
Let $(\CP^2,\cK)$ be a geometrically transverse, oriented surface of degree $d > 0$ and bridge index $b$.  Let $\Khat_{\lambda} = \cK \cap \widehat{Y}_{\lambda,N}$ for $N$ sufficiently large.  Then
\[sl(\Khat_1) + sl(\Khat_2) + sl(\Khat_3) = d^2 - 3d - b.\]
\end{proposition}

\begin{proof}
By Corollary \ref{cor:deg-freedom}, the three links have diagrams on the torus representing the classes
\begin{align*}
[\cS(K_1)] &= p [\alpha] + (d-q) [\beta] \\
[\cS(K_2)] &= q [\beta] + (d-r) [\gamma] \\
[\cS(K_3)] &= r [\gamma] + (d- p) [\alpha]
\end{align*}
for some $p,q,r$. The self-linking formula of Lemma \ref{lemma:torus-hat-sl}) yields
\begin{align*}
sl(\widehat{K}_1) & = w_1(\cS) + p(d-q) - p - (d-q) \\
sl(\widehat{K}_2) & = w_2(\cS) + q(d-r) - q - (d-r) \\
sl(\widehat{K}_3) & = w_3(\cS) + r(d - p) - r - (d-p) 
\end{align*}
Adding all three together, we obtain
\[
sl(\widehat{K}_1) + sl(\widehat{K}_2) + sl(\widehat{K}_3) = w_1(\cS) + w_2(\cS) + w_3(\cS) + d(p + q + r) - 3d - (pq + qr + rp).
\]
Conversely, from Proposition \ref{prop:framing-conservation} we also know that
\[ w_1(\cS) + w_2(\cS) + w_3(\cS) = -d(p + q + r) + (pq + qr + rp) + d^2 - b.\]
Substituting this into the previous equation yields the required equality.
\end{proof}
}
An easy corollary is that any surface in transverse bridge position satisfies the adjunction inequality.

\begin{proof}[Proof of Theorem \ref{thrm:transverse-Thom}]
Since $\cK$ is in transverse bridge position, each $K_{\lambda}$ is the $c_{\lambda}$-component unlink.  Thus the Bennequin bound implies that $sl(K_{\lambda}) \leq - c_{\lambda}$.  Summing over the three components and applying the formula of Proposition \ref{prop:sl-count}, we obtain the inequality
\[d^2 - 3d - b \leq -c_1 - c_2 - c_3\]
or equivalently
\[3d - d^2 \geq c_1 + c_2 + c_3 - b = \chi(\cK).\]
\end{proof}

%% file: braiding.tex
\section{Algebraic transverse bridge position}
\label{sec:alg-trans}

\subsection{Algebraic transversality}

Let $(\CP^2,\cK)$ be in general position with respect to the standard trisection.  Recall that $\cK$ is geometrically transverse if each oriented tangle $\tau_{\lambda}$ is positively transverse to the foliation of $H_{\lambda}$ by holomorphic disks.  We have polar coordinates {$(\theta_{\lambda+1}, r_{\lambda}, \theta_{\lambda})$} on $H_{\lambda}$.  The hyperplane field tangent to the foliation is the kernel of the 1-form $d \theta_{\lambda+1}$.  Therefore, the surface is geometrically transverse if and only if for each $\lambda$ we have that $d\theta_{\lambda+1}( \tau'_{\lambda})$ is everywhere positive.

\begin{definition}
A surface  $(\CP^2,\cK)$ is {\it algebraically transverse} if, for each $\lambda$ and each component $\tau_{\lambda,j}$ of $\tau_{\lambda}$, we have that
\[\int_{\tau_{\lambda,j}} d \theta_{\lambda+1} > 0.\]
An algebraically transverse surface in bridge position is in {\it algebraically transverse bridge position}.
\end{definition}

Clearly, if $(\CP^2,\cK)$ is geometrically transverse, it is algebraically transverse.  We will also refer to a fixed tangle $(H_{\lambda},\tau_{\lambda})$ as geometrically or algebraically transverse if it satisfies the { appropriate criterion.}

\begin{lemma}
Let $(\CP^2,\cK)$ be in algebraically transverse bridge position.  Then $\cK$ is regularly homotopic through algebraically transverse surfaces to a geometrically transverse surface.
\end{lemma}

\begin{proof}
{
The tangle $\tau_{\lambda}$ lives in $H_{\lambda} \simeq S^1 \times D^2$.  For any properly embedded arc $\tau$ satisfying $\int_{\tau_{\lambda,j}} d \theta_{\lambda+1} > 0$, { there is an arc $\widetilde{\tau}$ with the same endpoints that lifts to a (geometrically transverse) straight line in the universal cover $\RR \times D^2$} and satisfies
\[\int_{\tau} d \theta_{\lambda+1} = \int_{\widetilde{\tau}} d \theta_{\lambda+1}\]
The fundamental group of $H_{\lambda}$ is $\ZZ$.  The integral of $d \theta_{\lambda+1}$ is a homotopy invariant of properly embedded arcs with fixed endpoints; in fact, in $H_{\lambda}$ it determines the homotopy class rel endpoints.  Consequently, $\tau$ and $\widetilde{\tau}$ are homotopic rel boundary (in fact, regularly homotopic).  More generally, it follows that an algebraically transverse tangle is regularly homotopic, through algebraically transverse tangles, to a geometrically transverse tangle.  }

Given such a regular homotopy of $\tau_{\lambda}$ with fixed endpoints, we can extend this to a regular homotopy of the surface $\cK$.
\end{proof}

\begin{remark}
It is useful to note that homotoping $\tau_{\lambda}$ through a crossing change corresponds to a finger move of the surface $\cK$.  { To see this, choose local coordinates $H_{\lambda} \times \RR$ such that $\cK$ is the product of { $b$-component tangle} $\tau_{\lambda}$ with the $\RR$-factor and such that $\tau_{\lambda}$ is the image of $\phi: \coprod_{b} [0,1] \rightarrow H_{\lambda}$.  Now let $\phi_t$ be the homotopy of $\tau_{\lambda}$ through a crossing change.  We can extend this to a homotopy of $\cK$, such that afterwards, the intersection of $\cK$ with $H_{\lambda} \times [-1,0]$ is the trace of $\phi_t$ and its intersection with $H_{\lambda} \times [0,1]$ is the trace of the inverse homotopy. This introduces two self-intersections to $\cK$, one positive and one negative, corresponding to the moment of the self-intersection of $\tau_{\lambda}$ during the homotopy.}
\end{remark}

\subsection{Isotopy to algebraic transversality}

Every essential, oriented, embedded surface in $\CP^2$ is isotopic to a surface in algebraically transverse bridge position.

\begin{proposition}
\label{prop:alg-trans}
Let $(\CP^2,\cK)$ be an embedded, oriented, connected surface of positive degree.  Then $\cK$ can be isotoped into algebraically transverse bridge position.
\end{proposition}

\begin{proof}
We will fix absolute coordinates and work in the coordinate chart $z_3 = 1$.  Then we have polar coordinates
\[ z_1 = r_1 e^{i x} \qquad z_2 = r_2 e^{i y}\]
in this chart that induce coordinates $x,y \in [0,1]$ on the central surface $\Sigma$.  We can assume that the intersection of $\cK$ with the spine of the trisection lies in this coordinate chart and therefore evaluate algebraic transversality in these coordinates.  Recall that the foliation of $H_{\lambda}$ by holomorphic disks is determined by a 1-form $d \theta_{\lambda+1}$.  In the current coordinates, these 1-forms are
\[d \theta_2 = dy \qquad d \theta_3 = -dx \qquad d \theta_1 = dx - dy\]
By Theorem \ref{thrm:MZ-GBT}, we can assume $\cK$ is in bridge position.  Thus, each tangle $\tau_{\lambda}$ is boundary-parallel and we can assume that each projection $\pi_{\lambda}(\tau_{\lambda})$ is a collection of embedded arcs.

\begin{figure}[h]
\centering
\includegraphics[width=.4\textwidth]{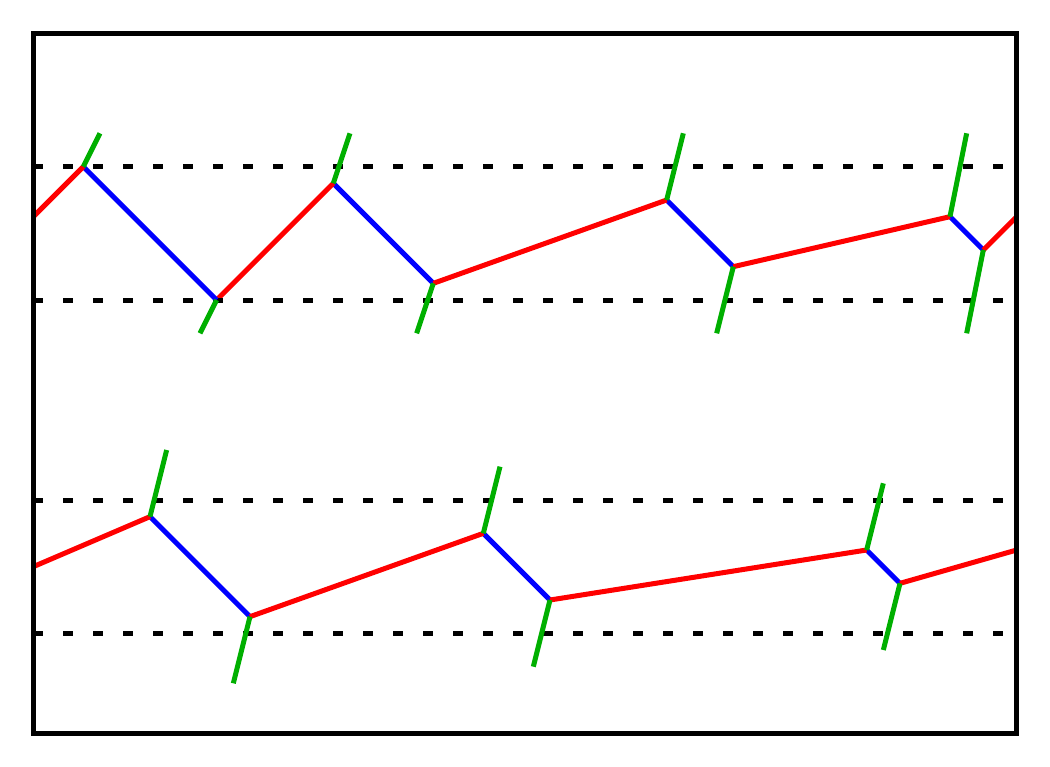}
\caption{A torus diagram for $\cK$, standardized with respect to $K_1$ and with $c_1 = 2$.  The content of the $\cC$-arcs is not depicted.}
\label{fig:epsilon-bands}
\end{figure}

First, we standardize the torus diagram with respect to $K_1$.  Since $K_1$ is the unlink with $c_1$ components, there is an isotopy so that $\cS(K_1)$ is embedded and with each component isotopic in $T^2$ to $\alpha$ and with the same orientation.  In particular, every bridge splitting of the unlink in $S^3$ is standard and so any two bridge splittings are isotopic \cite{MZ-Bridge-Trisection}.  Furthermore, by { braid stabilizations} of $\cC$, we can ensure that in homology
\begin{align*}
[\cS(K_1)] &= c_1[\alpha] + 0 [\beta] \\
[\cS(K_2)] &= d [\beta]  + 0 [\gamma] \\
[\cS(K_3)] &= d[\gamma] + (d - c_1) [\alpha] 
\end{align*}
We can assume that for any $\epsilon > 0$, each component of $\cS(K_1)$ lies in an open annulus $S^1 \times (y_0,y_0 + \epsilon)$ and that these annuli are pairwise disjoint.  Within each annulus, { we can isotope $\cB$ so that each }arc has slope $-1$.  Furthermore, we can assume that, projecting the $\cB$-arcs onto the curve $\{0\} \times S^1$, they are nested: for any pair of arcs $\cB_i,\cB_j$, either the endpoints of $\cB_i$ are contained in the interval between the endpoints of $\cB_j$, or vice versa.  Also, by a perturbation, we can assume that each bridge point has a unique $x$ and a unique $y$ coordinate.  An example standardized diagram is depicted in Figure \ref{fig:epsilon-bands}.

Since the height of the annulus is $\epsilon$, the length of the $\beta$ arcs is at most $\sqrt{2} \epsilon$ and their width in the horizontal direction is at most $\epsilon$.  Consequently, we must have that
\[
c_1 - b \epsilon < \int_{\cA} d x < c_1.
\]
By construction, we have that
\[\int_{\cB} dx \approx \int_{\cB} dy \approx 0,\]
so we can therefore conclude that 
\begin{align*}
\int_{\cA} dx &\approx c_1 & \int_{\cC} dx & \approx d \\
\int_{\cA} dy & \approx 0 & \int_{\cC} dy &\approx - d
\end{align*}
In particular, we have that $\Gamma \coloneqq \int_{\cC}(dx - dy) \approx 2d \gg 0$.  {Take $\RR^b$ with basis $\{e_i\}$.  The multiarc $\cC$ determines a vector 
\[\overline{\cC}  = \sum_{i = 1}^b \left(\int_{\cC_i} (dx - dy) \right) e_i \]}
This vector lies in the hyperplane $P_{\Gamma} \coloneqq \{x_1 + \dots + x_b = \Gamma\}$.  

\begin{figure}[h]
\centering
\labellist
	\large\hair 2pt
	\pinlabel $\cA_i$ at 190 100
	\pinlabel $\cC_{i_1}$ at 20 20
	\pinlabel $\cC_{i_2}$ at 350 250
\endlabellist
\includegraphics[width=.5\textwidth]{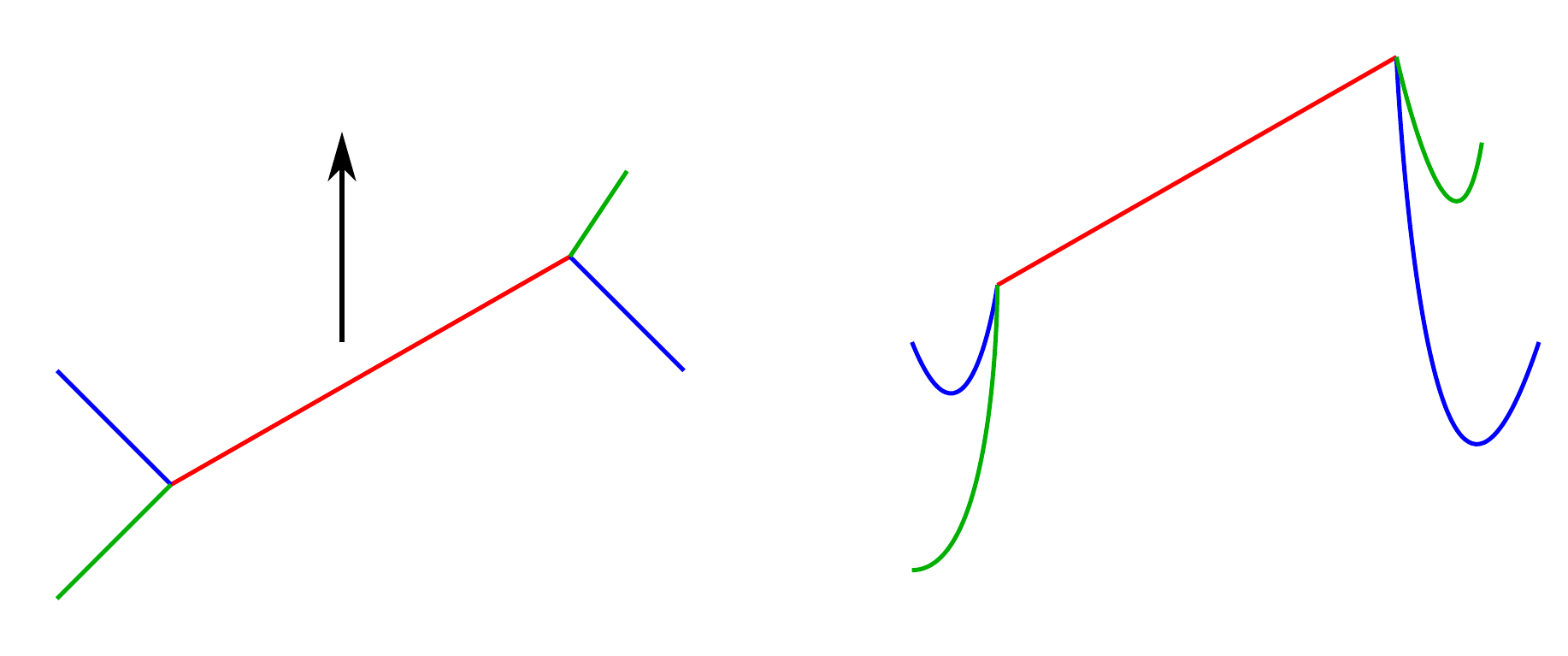}
\caption{({\it Left}) The arc $\cA_i$. ({\it Right}) The shadow diagram after translating the arc $\cA_i$ in the $y$-direction}
\label{fig:A-translate}
\end{figure}

Let $\cA_i$ be some arc of $\cA$ and let $\cC_{i_1},\cC_{i_2}$ be the incident arcs at the endpoints.  Translation in the $y$-direction of the entire arc $\cA_i$ by $M$ preserves the fact that $\cA_i$ is algebraically transverse, while increasing $\int_{\cC_{i_1}} (dx - dy)$ by $M$ and decreasing $\int_{\cC_{i_2}} (dx - dy)$ by $M$.  We thus translate the vector $\overline{\cC}$ in $P_{\Gamma}$ by $M(e_{i_1} - e_{i_2})$.  Similarly, let $\cB_{j}$ be some arc and $\cC_{j_1},\cC_{j_2}$ the incident arcs.  Each horizontal translation of $\cB_j$ by $M$ also translates $\overline{\cC}$ by $M(e_{j_1} - e_{j_2})$.  

{ Consider the subspace $\{x_1 + \dots + x_n = 0\}$.  It is spanned by vectors of the form $e_i - e_j$.  Observe that the vector $e_i - e_j$ lies in the span of $\{e_{i_1} - e_{i_2}, e_{j_1} - e_{j_2}\}$ if and only if there exists a path from $\cC_i$ to $\cC_j$ consisting of arcs of $\cA$ and $\cB$.  The spine of $\Sigma$ is connected, since $\Sigma$ is connected, so the collection of vectors $\{e_{i_1} - e_{i_2}, e_{j_1} - e_{j_2}\}$ span the subspace $\{x_1 + \dots + x_n = 0\}$.} We can therefore translate $\overline{\cC}$ to any point in $P_{\Gamma}$ by vertical translation of $\cA$-arcs and horizontal translation of $\cB$-arcs.  In particular, since $\Gamma > 0$, we can assume that $\overline{\cC}$ lies in the positive orthant --- i.e. that $\cC$ is algebraically transverse.
\end{proof}

\subsection{Braiding tangles in $H_{\lambda}$}
\label{sub:braiding}

Let $M$ be an oriented 3-manifold with $\del M = T^2$.  A {\it relative open book decomposition} on $M$ is a pair $(B,\rho)$ satisfying:
\begin{enumerate}
\item the {\it binding} $B$ is an oriented link { in the interior of $M$},
\item $\rho : M \smallsetminus B \rightarrow S^1$ is a fibration,
\item the restriction $\rho|_{\del M}: \del M \rightarrow S^1$ is a fibration,
\item the closure of each fiber $\rho^{-1}(\theta)$ is a compact, oriented surface $F_{\theta}$, whose boundary is the union of $B$ with a simple closed curve in $\del M$.
\end{enumerate}
The surfaces $\{F_{\theta}\}$ are the {\it pages} of the open book decomposition.  An oriented tangle $(M,\tau)$ is {\it braided} with respect to a relative open book $(B,\pi)$ if the tangle is everywhere positively transverse to the pages.  In particular, the tangle is disjoint from the binding $B$.  

On each solid torus $H_{\lambda}$ of the trisection of $\CP^2$, we construct a specific relative open book decomposition as follows.  Recall that we have polar coordinates { $(\theta_{\lambda+1},r_{\lambda},\theta_{\lambda})$}.  The binding is the core circle $B_{\lambda} \coloneqq { \{ r_{\lambda} = 0\}}$.  On the complement of $B_{\lambda}$, we have a fibration
\[\rho_{\lambda}: { (\theta_{\lambda+1}, r_{\lambda},\theta_{\lambda}) \mapsto \theta_{\lambda}.}\]
The pair $(B_{\lambda},\rho_{\lambda})$ is then a relative open book decomposition of $H_{\lambda}$.  

Thus, we now have two notions of positivity for tangles in $H_{\lambda}$: geometric transversality and braiding.  These correspond to two distinct ways of foliating $H_{\lambda}$ and $H_{\lambda} \smallsetminus B_{\lambda}$.  See Figure \ref{fig:positive-tangle}.  { They can also be detected in a torus diagram; the following lemma follows by definition.

\begin{lemma}
The tangle $\tau_{\lambda}$ is geometrically transverse if and only if $\cA_{\lambda}$ moves monotonically positively with respect to the $\theta_{\lambda+1}$-coordinate.  The tangle $\tau_{\lambda}$ is braided if and only if $\cA_{\lambda}$ moves monotonically positively with respect to the $\theta_{\lambda}$-coordinate.
\end{lemma}
}

Adapting the proof of Alexander's Theorem, we can braid any tangle with respect to this open book decomposition.

\begin{proposition}
Let $(H_{\lambda},\tau_{\lambda})$ be a tangle.  There is an isotopy of $\tau_{\lambda}$ such that it is braided with respect to the relative open book decomposition $(B_{\lambda},\rho_{\lambda})$.
\end{proposition}

{
\begin{proof}
Assume that $\tau_{\lambda}$ is disjoint from $B_{\lambda}$.  The tangle $\tau_{\lambda}$ is braided if and only if $d\rho_{\lambda}(\tau'_{\lambda}) > 0$ everywhere.  First, we can isotope $\tau_{\lambda}$ to satisfy this condition near the endpoints.  After possibly a perturbation, we have a finite number of segments along which $\tau_{\lambda}$ fails to be braided.  However, by isotoping these segments through the braid axis $B_{\lambda}$ we can braid the tangle.
\end{proof}
}

\begin{figure}[h]
\centering
\includegraphics[width=.4\textwidth]{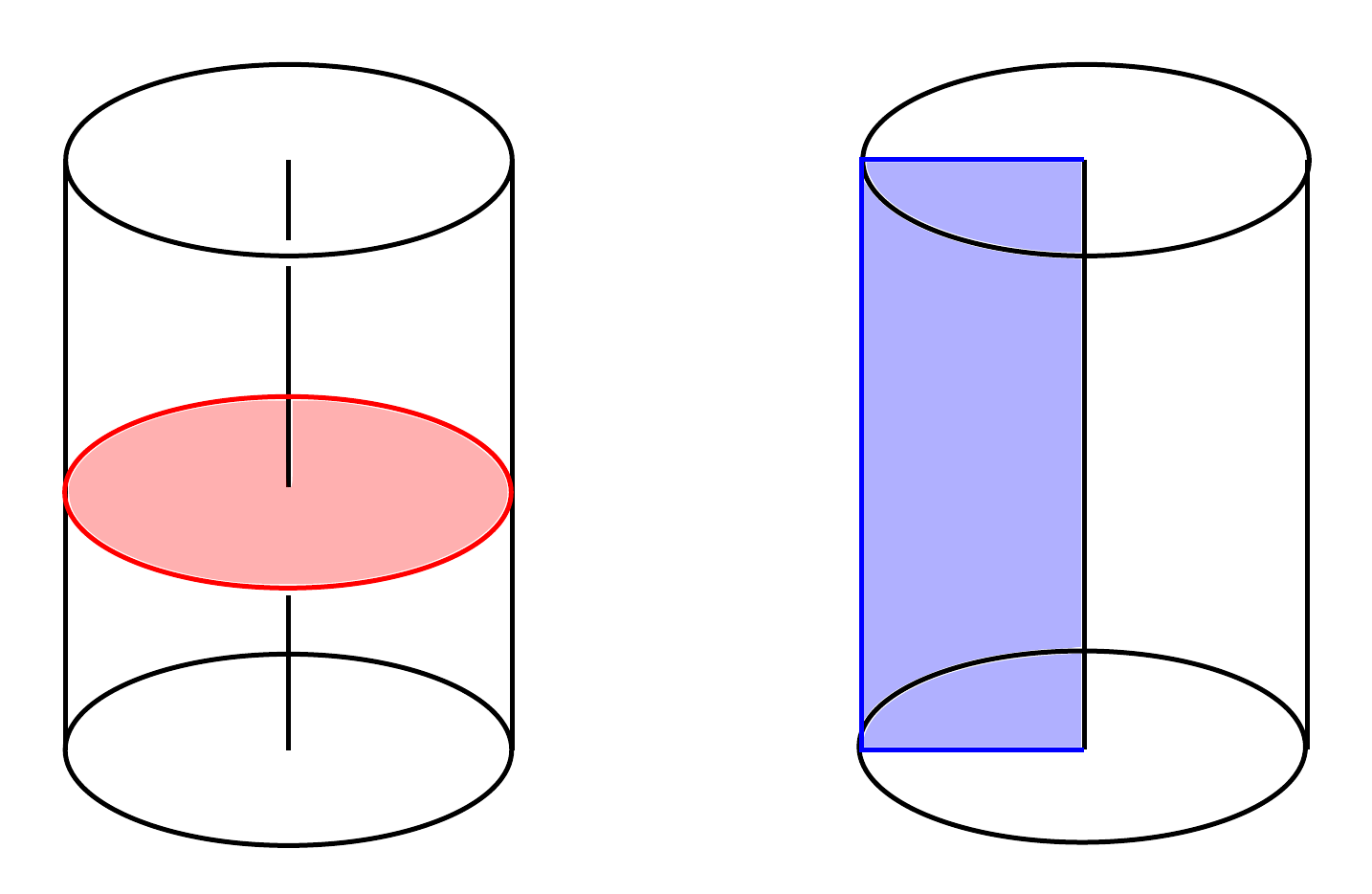}
\caption{({\it Left}) Foliation of $H_{\lambda}$ by holomorphic disks.  ({\it Right}) A relative open book decomposition of $H_{\lambda}$.}
\label{fig:positive-tangle}
\end{figure}

The following lemma will be useful in the sequel.

\begin{lemma}
\label{lemma:mini-bridge-stab}
Mini bridge stabilization along a geometrically transverse, braided arc preserves geometric transversality of all three tangles $\tau_{1},\tau_{2},\tau_{3}$.
\end{lemma}

\begin{proof}
The proof is Figure \ref{fig:mini-stabilization}.  { In particular, choose the orientation on $\cA$ so that is moving up (i.e. geometrically transverse) and to the right (i.e. braided).  Then after a mini-stabilization, the new $\cB$-arc is oriented from $-$ to $+$ and is moving monotonically to the left, so is geometrically transverse, while the new $\cC$-arc is moving monotonically down with respect to lines of slope 1 in the picture, so it is also geometrically transverse.}
\end{proof}

\subsection{Simple clasps}

The prototypical example of a tangle in $H_{\lambda}$ that is algebraically transverse and braided, but not geometrically transverse, is the simple clasp in Figure \ref{fig:simple-clasp}.  Each of the two strands $\tau_1,\tau_2$ can individually be isotoped to be geometrically transverse, but not both simultaneously.

In fact, these clasps are the entire obstruction to isotoping a braided, algebraically transverse tangle $\tau_{\lambda}$ into a geometrically transverse tangle.  In the former, each arc $\tau_i$ is isotopic to an arc $\upsilon_i$ that projects to a straight line of positive slope on $T^2$.  By `pulling tight', we can attempt to make $\tau$ geometrically transverse by moving each $\tau_i$ towards $\upsilon_i$.  The obstruction is clearly a collection of clasps.  We can then apply mini bridge stabilizations to separate the clasps from one another.

\begin{definition}
A {\it simple clasp} is a tangle $\tau$ in $H_{\lambda}$ consisting of two arcs $\tau_1,\tau_2$ in $\tau$ and a Whitney disk $W$ satisfying
\begin{enumerate}
\item $\tau$ is algebraically transverse and braided, with $\tau_1$ geometrically transverse;
\item the Whitney disk $W$ intersects $\tau_1$ transversely in a single point;
\item the boundary of $W$ is the union of two arcs $\widehat{\tau}_2$ and $\widehat{\upsilon}_2$, where $\widehat{\tau}_2$ is a connected subarc of $\tau_2$ and $\widehat{\upsilon}_2$ is a geometrically transverse arc.
\end{enumerate}
\end{definition}

In addition, we define the {\it degree} of a simple clasp to the maximum cardinality of $\pi_{\lambda}^{-1} \circ \pi_{\lambda}(x)$ for any $x \in W$.  An example of a simple clasp of degree 3 is given in Figure \ref{fig:compound-clasp}.

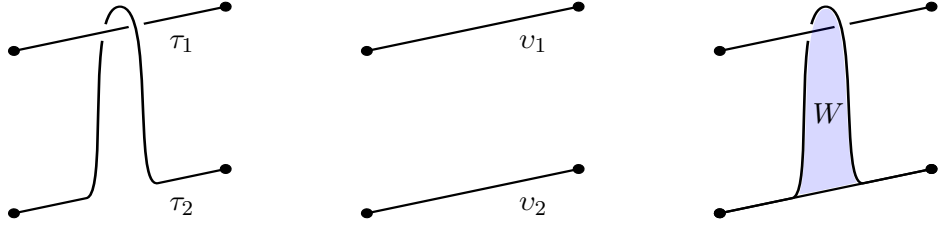
\begin{figure}[h]
\centering
\labellist
	\large\hair 2pt
	\pinlabel $\tau_1$ at 190 170
	\pinlabel $\tau_2$ at 190 40
	\pinlabel $\upsilon_1$ at 470 170
	\pinlabel $\upsilon_2$ at 470 40
	\pinlabel $W$ at 705 115
\endlabellist
\includegraphics[width=.9\textwidth]{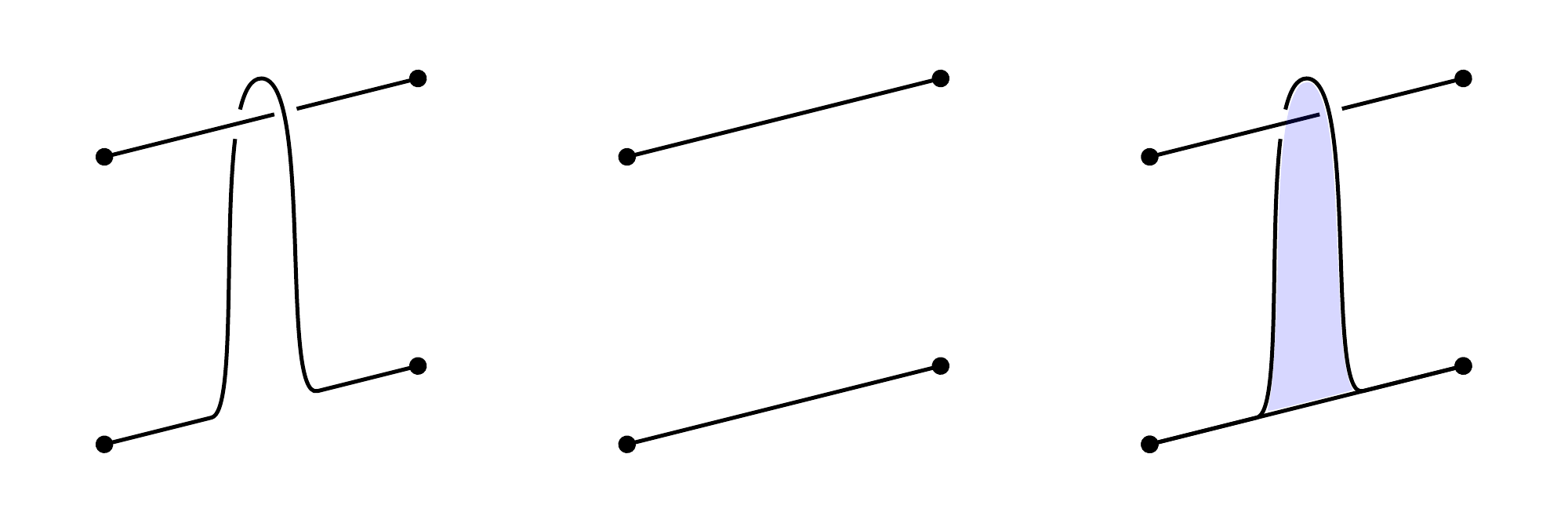}
\caption{({\it Left}) A simple clasp $\tau = \tau_1 \cup \tau_2$. ({\it Middle}) A simple clasp $\tau$ is homotopic to geometrically transverse tangle $\upsilon = \upsilon_1 \cup \upsilon_2$. ({\it Right}) The Whitney disk $W$ directs the homotopy from $\tau$ to $\upsilon$. }
\label{fig:simple-clasp}
\end{figure}

\begin{proposition}
\label{prop:simple-clasps}
Let $\tau$ be a tangle in $H_{\lambda}$ that is algebraically transverse and braided.  Then by a sequence of isotopies and mini bridge stabilizations, we can assume that each arc of $\tau$ is either
\begin{enumerate}
\item geometrically transverse and braided, or
\item one of two arcs in a simple clasp.
\end{enumerate}
\end{proposition}

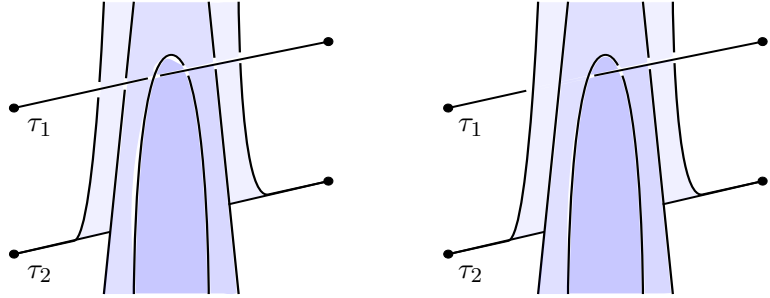
\begin{figure}[h]
\centering
\labellist
	\large\hair 2pt
	\pinlabel $\tau_1$ at 55 190
	\pinlabel $\tau_2$ at 55 45
	\pinlabel $\tau_1$ at 480 190
	\pinlabel $\tau_2$ at 480 45
\endlabellist
\includegraphics[width=.7\textwidth]{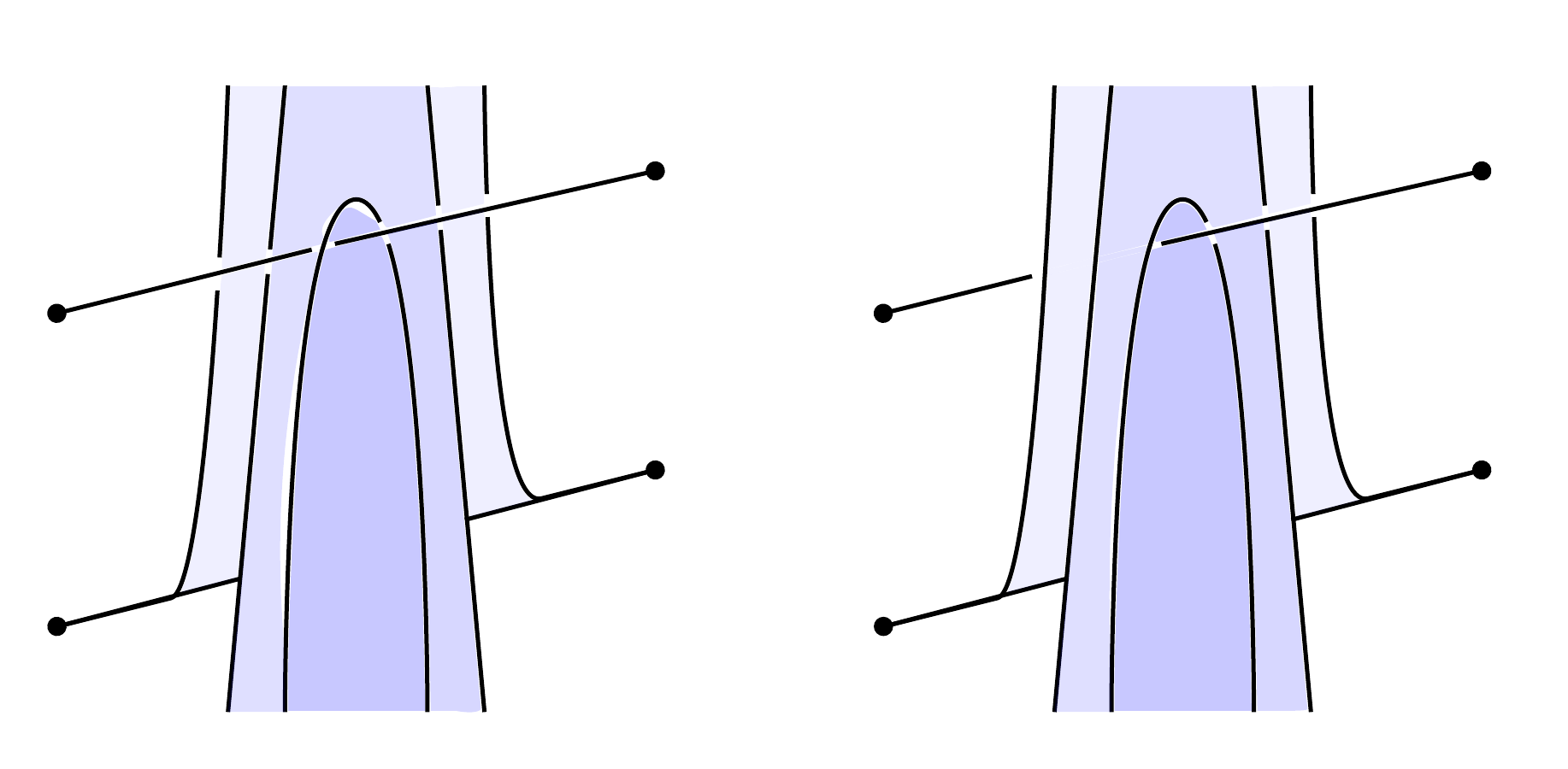}
\caption{({\it Left}) A simple clasp of degree 3. ({\it Right}) This clasp is not simple, as the Whitney disk intersects $\tau_1$ in three points. }
\label{fig:compound-clasp}
\end{figure}

\begin{figure}
\centering
\labellist
	\small\hair 2pt
	\pinlabel $\tau_1$ at 45 200
	\pinlabel $\tau_2$ at 45 25
	\pinlabel $W$ at 130 100
	\pinlabel $W_1$ at 300 100
	\pinlabel $W_1$ at 670 100
	\pinlabel $W_2$ at 823 100
	\pinlabel $W_3$ at 992 100
\endlabellist
\includegraphics[width=.99\textwidth]{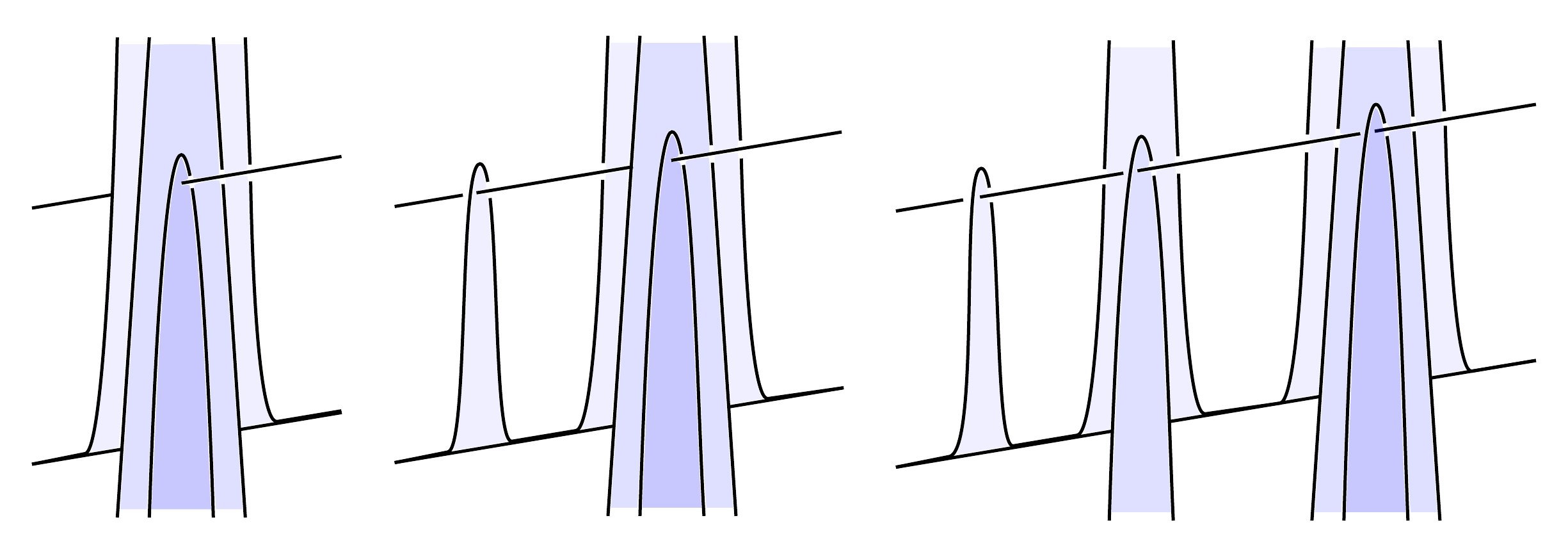}
\caption{({\it Left}) Initially, $W$ is a Whitney disk that intersects $\tau_1$ in three points.  ({\it Middle}) By an isotopy that projects to a Reidemeister II move, we can split off a Whitney disk $W_1$ that intersects $\tau_1$ once.  ({\it Right}) Repeating, we can split $W$ into three Whitney disks, leaving three simple clasps. }
\label{fig:compound-to-simple}
\end{figure}

\begin{proof}[Proof of Proposition \ref{prop:simple-clasps}]
As described above, we can achieve the proposition by `pulling tight' and then stabilizing.  However, we should think of the tangle as living in $I \times T^2 = H_{\lambda} \smallsetminus \nu(B_{\lambda})$, which is not simply-connected.  Therefore, there is some subtlety in the fact that arcs may wind around the torus before clasping.  We therefore give a careful exposition of this procedure.

Let $\tau_i$ be an arc of $\tau$ and choose a lift $\widehat{\tau}_i$ to $I \times \RR^2$.  The coordinates $\theta_{\lambda},\theta_{\lambda+1}$ pull back to coordinates $x,y$ on $\RR^2$.  { `Pulling tight' can be achieved by an isotopy rel endpoints in the $y$-direction.  } Let $\widehat{\upsilon}_i$ be the arc obtained by isotoping $\widehat{\tau}_i$ rel endpoints in the { $y$}-direction until its projection to $\RR^2$ is a straight line and let $\upsilon_i$ be the image of $\widehat{\upsilon}_i$ in $I \times T^2$.  By a perturbation of the bridge points, we can assume the union $\upsilon = \cup \{\upsilon_i\}$ is embedded in $I \times T^2 $.  By an isotopy of $\tau$ in the $y$-direction in $I \times T^2$, we can assume that each point of each arc $\widehat{\tau}_i$ has greater $y$-coordinate in $I \times \RR^2$ than its corresponding point in $\widehat{\upsilon}_i$.  Finally, we can assume that the projection of $\tau$ to the strip $I \times \RR$, obtained by forgetting the $y$-coordinate, is generic with a finite number of transverse double points.

Now, we can pull the tangle tight by isotoping $\tau_i$ in the negative $y$-direction until it agrees with $\upsilon_i$.  Clearly, the only potential obstructions lie above the crossings in the projection to $I \times \RR$.  We can now assume that $\tau$ and $\upsilon$ agree outside an arbitrarily small neighborhood of these potential obstructions.  Let $\tau_i,\tau_j$ be two arcs that project to a crossing.   Let $p_i,q_j$ denote the preimages of the crossing in $\tau_i$ and $\tau_j$, respectively.  (Note that we may have $i = j$, provided that the points $p_i,q_j$ are distinct).  { Choose lifts $\widehat{\tau}_i,\widehat{\upsilon}_i$ with the same endpoints and let $R_i$ be the difference between the $y$-coordinates of $\widehat{\tau}_i$ and $\widehat{\upsilon}_i$ at $p_i$.  The constant $R_i$ is independent of the chosen lifts as all choices are related by translating in the $y$-direction by $2 \pi n$ for some integer $n$.  Similarly, define $S_j$ to be the difference between the $y$-coordinates of $\widehat{\tau}_j$ and $\widehat{\upsilon}_j$ at $q_j$, for any choice of lifts of $\tau_j$ and $\upsilon_j$.  Let $\Delta$ be the distance in the positive $\theta_{\lambda+1}$ direction from $\tau_j(p_j)$ to $\tau_i(q_i)$ and let $\Delta' = 2 \pi - \Delta$ be the distance in the negative $\theta_{\lambda+1}$ direction.
}

We now have three cases:
\begin{enumerate}
\item if { $R_i - \Delta < S_j < R_i + \Delta'$}, then we can simultaneously isotope $\tau_i$ to agree with $\upsilon_i$ and $\tau_j$ to agree with $\upsilon_j$.
\item if { $S_j + \Delta < R_i $}, then to isotope $\tau_i$ to $\upsilon_i$ we must homotope it through $\tau_j$ at least once.  In fact, the number of times we must pass $\tau_i$ through $\tau_j$ is exactly
{ \[ k \coloneqq \lceil R_i -  S_j - \Delta\rceil.\] }
\item if { $R_i + \Delta' < S_j $}, then to isotope $\tau_j$ to $\upsilon_j$ we must homotope it through $\tau_i$ at least once.  Again, the number of times we must pass $\tau_j$ through $\tau_i$ is exactly
{ \[ k \coloneqq \lceil S_j - R_i - \Delta' \rceil.\] }
\end{enumerate}

{ Given case (1), we are done.  Thus,} without loss of generality, by reordering the arcs $\tau_i,\tau_j$ we can assume we are in case (2). We isotope $\tau_j$ to agree with $\upsilon_j$ and isotope $\tau_i$ as far as possible until it locally forms a clasp with $\tau_j$ as in Figure \ref{fig:simple-clasp}.  The vertical lines between corresponding points in $\tau_i$ and $\upsilon_i$ comprise a Whitney disk $W$, which we can perturb to be embedded.  An example with $k = 3$ is given on the right of Figure \ref{fig:compound-clasp}.

Note that $W$ intersects $\tau_j$ exactly $k$ times, once for each crossing change necessary to isotope $\tau_i$ to $\upsilon_i$.  Thus we do not yet have simple clasps.  However, by an isotopy and mini stabilizations, we can replace this tangle with $k$ simple clasps, one of each degree from $1$ to $k$.  Choose $k -1 $ properly embedded arcs that cut $W$ into $k$ disks, each containing one intersection point with $\tau_j$.  Isotope $\tau_i$ along these arcs to agree with $\upsilon_i$.  This splits $W$ into $k$ disks, each corresponding to a single simple clasp.  See Figure \ref{fig:compound-to-simple}.
\end{proof}

%% file: rbi.tex
\section{Ribbon-Bennequin Inequality}

To prove the adjunction inequality, we need to extend the ribbon-Bennequin inequality to transverse links in $\#_k (S^1 \times S^2, \xi_{std})$.  { This result can be deduced  from the generalized slice-Bennequin inequality of Akbulut-Matveyev \cite{Akbulut-Matveyev} or Lisca-Mati\'c \cite{Lisca-Matic}.  However, the aim of this paper is to use only 3-dimensional techniques.  In this section, we will use contact topology to reduce this general case to the ribbon-Bennequin inequality in $(S^3,\xi_{std})$.}

Let $K$ be a Legendrian knot in $(M,\xi)$.  In any neighborhood of $K$, we can find a tubular neighborhood $\nu(K)$ and a contactomorphism that identifies $\nu(K)$ with a neighborhood of the 0-section in $J^1(S^1)$.  This identification determines a framing of $K$, called the {\it contact framing}.  { Let $M_-(K)$ denote Dehn surgery on $M$ along $K$ with surgery slope $(- 1)$ relative to the contact framing.  There is a unique extension $\xi_-(K)$ of the contact structure $\xi$, restricted to $M \smallsetminus \nu(K)$, to $M_-(K)$ such that the restriction of $\xi_-(K)$ to the surgery solid torus is tight.}  We refer to $(M_-(K),\xi_-(K))$ as {\it Legendrian surgery} along $K$.

\begin{proposition}
\label{prop:leg-surgery-s1s2}
Let $U$ be a $k$-component Legendrian link in $(\#_k S^1 \times S^2, \xi_{std})$ whose $i^{\text{th}}$-component is smoothly isotopic to $S^1 \times \{pt\}$ in the $i^{\text{th}}$-factor.  The result of Legendrian surgery on $U$ is $(S^3,\xi_{std})$.
\end{proposition}

{
\begin{proof}
{ As mentioned above, we present a purely 3-dimensional proof.  Experts will easily think of a simpler, 4-dimensional proof.}

It follows from the assumptions that we can choose a collection $S_1 \cup \dots \cup S_k$ of essential spheres such that the geometric intersection number of the $j^{\text{th}}$ sphere $S_j$ with the $i^{\text{th}}$ component of $U$ is exactly 1 if $i = j$ and 0 otherwise.  Let $U_i$ denote the $i^{\text{th}}$-component.  Take a small tubular neighborhood $\nu(U_i \cup S_i)$.  It is diffeomorphic to $S^1 \times S^2$ with a small ball cut out and its boundary is a reducing sphere.  By a $C^{\infty}$-small perturbation, we can assume the boundary sphere is convex.  The restriction of $\xi_{std}$ to $\nu(U_i \cup S_i)$ is tight.  In particular, it is contactomorphic to the unique tight contact structure on $S^1 \times S^2$ with a Darboux ball cut out.  Legendrian surgery turns this into a standard Darboux ball.  The complement of all these neighborhoods is $S^3$ with $k$ balls cut out and $\xi_{std}$ restricts to the standard tight contact structure.  Thus, after performing all $k$ surgeries, we are left with $S^3$ with $k$ Darboux balls cut out, then reglued in.  Thus, we have $(S^3,\xi_{std})$.
\end{proof}
}

{ If $F$ is an immersed ribbon surface bounded by a link $L \subset Y$, we let $\chi(F)$ denote the Euler characteristic of the abstract source of the immersion, not the Euler characteristic of its image as a subset of $Y$.}
\begin{lemma}
\label{lemma:replace-s1s2}
Let $L$ be a transverse link in $\#_k (S^1 \times S^2, \xi_{std})$ that bounds a ribbon surface $F$.  There exists a transverse link $L'$ in $(S^3,\xi_{std})$ that bounds a ribbon surface $F'$ such that $sl(L) = sl(L')$ and $\chi(F) = \chi(F')$.
\end{lemma}

\begin{proof}
Let $U$ be the $k$-component link in $\#_k S^1 \times S^2$ whose $i^{\text{th}}$-component represents $S^1 \times \{pt\}$ in the $i^{\text{th}}$-factor.  By an isotopy, we can assume that $U$ is disjoint from the ribbon surface $F$ and then Legendrian realize it by a $C^0$-small perturbation.  By Proposition \ref{prop:leg-surgery-s1s2}, the result of Legendrian surgery along $U$ is $(S^3,\xi_{std})$.  We can perform Legendrian surgery in an arbitrary neighborhood of $U$. In particular, the neighborhood can be assumed disjoint from $L$ and $F$.  Let $L',F'$ be the images of $L,F$ after Legendrian surgery.  Note that we can resolve the self-intersections of $F$ to obtain a Seifert surface of $L$ in an arbitrary neighborhood of $F$.  Using this Seifert surface to compute the self-linking number before and after surgery, we see that $sl(L) = sl(L')$.
\end{proof}

\begin{theorem}
\label{thrm:rbi-s1s2}
Let $L$ be a transverse link in $\#_k (S^1 \times S^2,\xi_{std})$ and let $F$ be a ribbon surface bounded by $L$.  Then
\[sl(L) \leq - \chi(F).\]
\end{theorem}

\begin{proof}
If the pair $L,F$ violates the ribbon-Bennequin inequality in $\#_k (S^1 \times S^2,\xi_{std})$, then by Lemma \ref{lemma:replace-s1s2}, we can find a pair $L',F'$ in $(S^3,\xi_{std})$ that violates Theorem \ref{thrm:rbi}.
\end{proof}

%% file: adjunction.tex
\section{Adjunction inequality}

Let $(\CP^2,\cK)$ be a surface satisfying the conclusions of Proposition \ref{prop:simple-clasps}. Specifically, $\cK$ is in algebraically transverse bridge position and each arc of each oriented tangle $\tau_{\lambda}$ is either (1) geometrically transverse, or (2) one half of a simple clasp.  By a regular homotopy that undoes each of the simple clasps, we can replace $(\CP^2,\cK)$ with an immersed surface $(\CP^2,\cL)$ that is geometrically transverse.  In particular, each arc of each tangle $\upsilon_{\lambda} = \cL \cap H_{\lambda}$ is transverse to the foliations by holomorphic disks.  Furthermore, this regular homotopy does not change the bridge index.

Set $(Y,\xi) = (Y_1,\xi_1) \sqcup (Y_2,\xi_2) \sqcup (Y_3,\xi_3)$.  Taking disjoint unions, we also obtain links $K = K_1 \sqcup K_2 \sqcup K_3$ and $L = L_1 \sqcup L_2 \sqcup L_3$ in $Y$, where $L$ is transverse to $\xi$.  For each simple clasp of $\tau_{\lambda}$, choose a point $x \in H_{\lambda} \subset Y_{\lambda}$ in a tubular neighborhood of the Whitney disk.  Let $\overline{x} \in -H_{\lambda} \subset Y_{\lambda - 1}$ be its image in the mirror handlebody.  Let $\widetilde{Y}$ be the 3-manifold obtained by surgery on the 0-sphere $\{x \cup \overline{x}\}$ for each simple clasp.  Each 0-sphere $\{x \cup \overline{x}\}$ is trivially isotropic in $(Y,\xi)$, thus we can perform contact surgery to obtain a contact 3-manifold $(\widetilde{Y},\widetilde{\xi})$.

\begin{lemma}
The contact structure $(\widetilde{Y},\widetilde{\xi})$ is tight.
\end{lemma}

\begin{proof}
Disjoint union preserves tightness, so $(Y,\xi)$ is tight.  Colin proved that contact 0-surgery preserves tightness \cite{Colin}, thus the resulting contact structure $(\widetilde{Y},\widetilde{\xi})$ is tight.
\end{proof}

\begin{remark}
The contact manifold $(\widetilde{Y},\widetilde{\xi})$ has a 4-dimensional interpretation.  Set 
\[ (\widehat{Y},\widehat{\xi}) = (\widehat{Y}_{1,N},\widehat{\xi}_{1,N}) \sqcup (\widehat{Y}_{2,N},\widehat{\xi}_{1,N}) \sqcup (\widehat{Y}_{3,N},\widehat{\xi}_{1,N}).\]
This is the boundary of a Stein domain $\widehat{X}_{N} \coloneqq \widehat{X}_{1,N} \sqcup \widehat{X}_{2,N} \sqcup \widehat{X}_{3,N}$ whose complement in $\CP^2$ is a neighborhood of the spine of the trisection.

For each pair of points $x,\overline{x}$, we can choose a properly embedded arc in $\CP^2 \smallsetminus \widehat{X}_N$ whose boundary are the corresponding points in $\widehat{Y}_{\lambda,N}$ and $\widehat{Y}_{\lambda-1,N}$.  This arc is necessarily isotropic and therefore we can attach a Stein 1-handle to $\widehat{X}_N$ whose core is this arc.  For details, see \cite[Section 8.2]{CE-Stein}.  After attaching all of the 1-handles, we obtain a Stein domain whose boundary is $(\widetilde{Y},\widetilde{\xi})$.
\end{remark}

Let $\widetilde{K}$ be the link obtained from $K$ by adding $2n$ untwisted, symmetric bands near the simple clasps as in Figure \ref{fig:bands}.  Specifically, if two arcs $\tau_1,\tau_2$ form a simple clasp, we can find an untwisted, symmetric band connecting each pair $\tau_i,-\tau^r_i$ that runs across the 2-sphere created by surgery on $x \cup \overline{x}$.  Resolving this band produces a new 4-component tangle.  Repeating for all simple clasps yields the link $\widetilde{K}$.  The same bands exist connecting $\upsilon_1,\upsilon_2$ to their mirrors.  Let $\widetilde{L}$ be the resulting link.

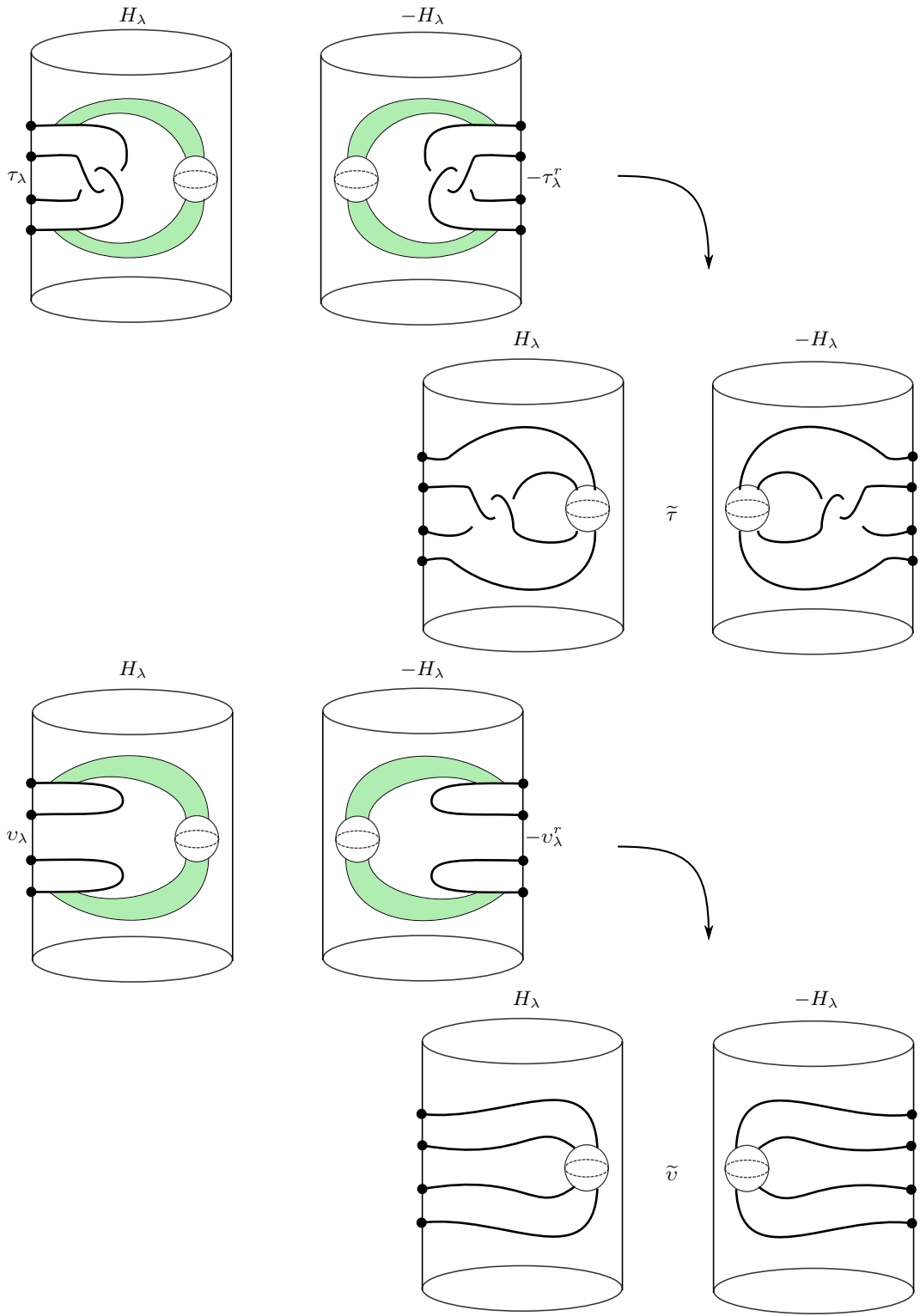
\begin{figure}
\centering
\labellist
	\small\hair 2pt
	\pinlabel $H_{\lambda}$ at 200 2250
	\pinlabel $-H_{\lambda}$ at 700 2250
	\pinlabel $\upsilon_{\lambda}$ at 0 830
	\pinlabel $-\upsilon^r_{\lambda}$ at 910 830
	\pinlabel $\tau_{\lambda}$ at 0 1970
	\pinlabel $-\tau^r_{\lambda}$ at 910 1970
	\pinlabel $H_{\lambda}$ at 200 1120
	\pinlabel $-H_{\lambda}$ at 700 1120
	\pinlabel $-H_{\lambda}$ at 1380 1690
	\pinlabel $H_{\lambda}$ at 880 1690
	\pinlabel $\widetilde{\tau}$ at 1130 1390
	\pinlabel $H_{\lambda}$ at 880 550
	\pinlabel $-H_{\lambda}$ at 1380 550
	\pinlabel $\widetilde{\upsilon}$ at 1130 250
\endlabellist
\includegraphics[width=.9\textwidth]{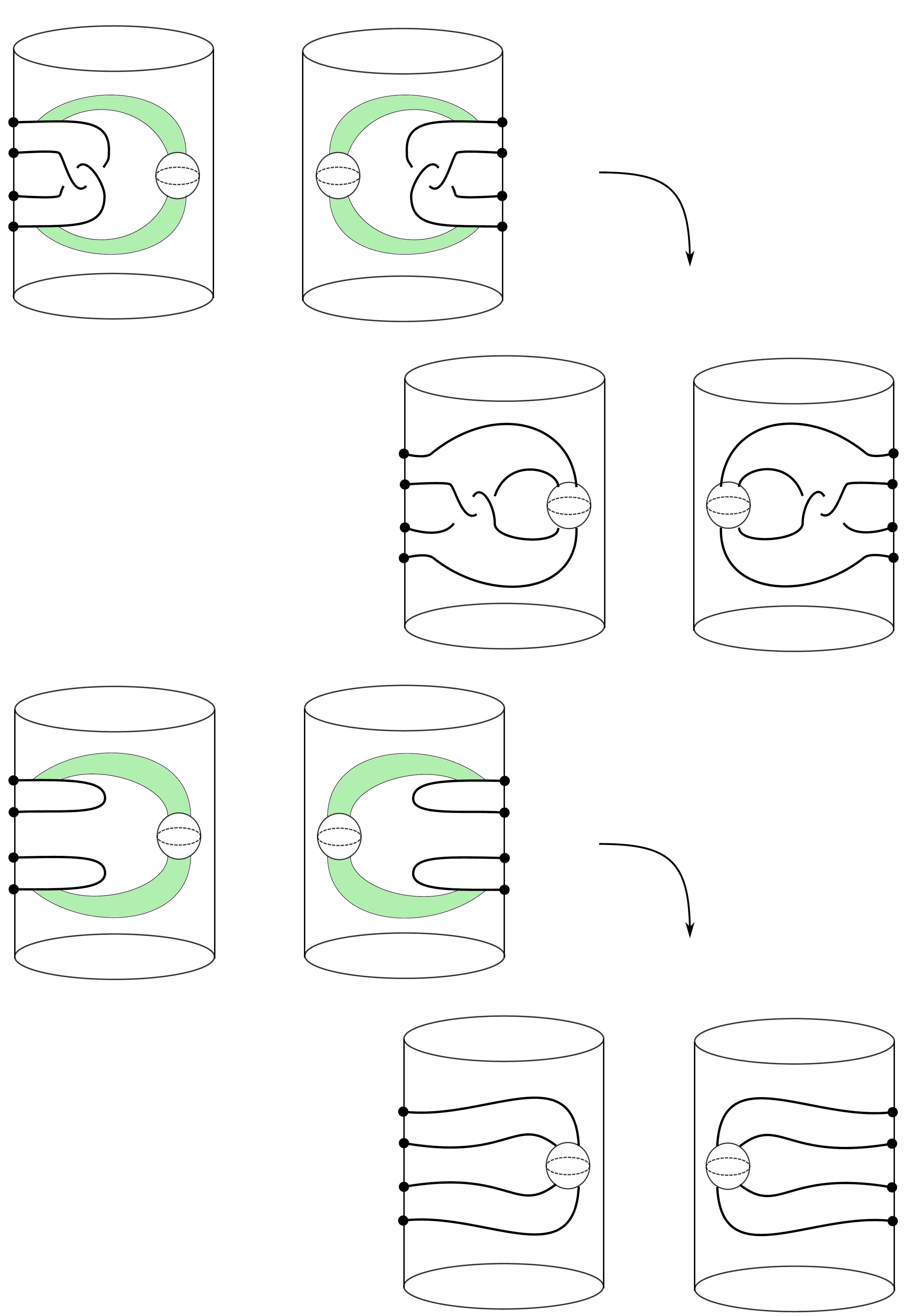}
\caption{The effect of attaching symmetric bands in $H_{\lambda} \cup -H_{\lambda}$. ({\it Top Left}) Bands attached to $\tau \cup -\tau^r$. ({\it Middle Right}) The resulting tangle $\widetilde{\tau}$. ({\it Middle Left}) Bands attached to $\upsilon \cup - \upsilon^r$. ({\it Bottom Right}) The resulting tangle $\widetilde{\upsilon}$. The tangles $\widetilde{\tau}$ and $\widetilde{\upsilon}$ are isotopic.}
\label{fig:bands}
\end{figure}

\begin{proposition}
\label{prop:Ktilde-ribbon}
Let $\widetilde{K},\widetilde{L}$ be the links obtained from $K$ and $L$, respectively, by adding $2n$ bands.
\begin{enumerate}
\item The links $\widetilde{K}$ and $\widetilde{L}$ are isotopic.
\item The link $\widetilde{K}$ bounds a ribbon surface $F$ with
\[\chi(F) = c_1 + c_2 + c_3 - 2n.\]
\end{enumerate}
\end{proposition}

\begin{proof}
That the first statement is true near each simple clasp can be seen in Figure \ref{fig:bands}.  Move all of the crossings of $\widetilde{\tau}$ in $H_{\lambda}$ and then cancel by two Reidemeister II moves.  Moreover, this local model occurs by assumption in a neighborhood of the Whitney disk, so we can simultaneously realize the isotopy at each simple clasp.

Secondly, the image of the link $K$ in $\widetilde{Y}$ is the unlink with $c_1 + c_2 + c_3$ components.  Therefore it bounds a collection of disjoint, embedded disks.  Then $\widetilde{K}$ is obtained by surgering $2n$ bands to the unlink $K$ and the surface $F$ is the union of the original Seifert disks with these bands.
\end{proof}

\begin{proposition}
\label{prop:Ltilde-sl}
The link $\widetilde{L}$ admits a transverse representative with self-linking number
\[sl(\widetilde{L}) = d^2 - 3d - b + 2n.\]
\end{proposition}

\begin{proof}
The surface $\cL$ is geometrically transverse, so the link $L = L_1 \sqcup L_2 \sqcup L_3$ has self-linking number
\[sl(L) = d^2 - 3d - b\]
by Proposition \ref{prop:sl-count}.  To prove the statement, we need to show that each band can be attached to $L$ so that the result is transverse and so that each band increases the self-linking number by 1.

Topological 0-surgery on $x \cup \overline{x}$ is performed by cutting out $B^3$ neighborhoods of $x$ and $\overline{x}$ and then identifying their boundaries.  To perform this 0-surgery in the contact category, the boundaries must be convex surfaces with diffeomorphic characteristic foliations.  For a thorough description, see \cite[Section 4.12]{Geiges}.  

Take polar coordinates $(\theta_{\lambda+1},r_{\lambda},\theta_{\lambda})$ on $H_{\lambda}$.  Let $-H_{\lambda}$ denote $H_{\lambda}$ with the opposite orientation.  Let $\Phi_{\lambda}: H_{\lambda} \rightarrow -H_{\lambda}$ be the identity map, which is an orientation-reversing diffeomorphism.  We view $H_{\lambda}$ as a subset of $Y_{\lambda}$ and so it inherits the contact structure $\xi_{\lambda}$.  A contact form defining the restriction of $\xi_{\lambda}$ to $H_{\lambda}$ is
\[\alpha_+ = d \theta_{\lambda+1} + h_+(r) d \theta_{\lambda}\]
for some suitable increasing function $h_+(r)$ satisfying $h_+(0) = 0$.  In addition, we view $-H_{\lambda}$ as a subset of $Y_{\lambda - 1}$ and it inherits the contact structure $\xi_{\lambda - 1}$.  Along the core $B_{\lambda} = \{r_{\lambda} = 0\}$ of the solid torus, the contact structure $\xi_{\lambda - 1}$ is tangent to the foliation by holomorphic disks.  Consequently, a contact form defining the restriction of $\xi_{\lambda - 1}$ to $-H_{\lambda}$ is
\[ \alpha_- = -d \theta_{\lambda + 1} + h_-(r) d \theta_{\lambda}\]
for some suitable increasing function $h_-(r)$ satisfying $h_-(0) = 0$.

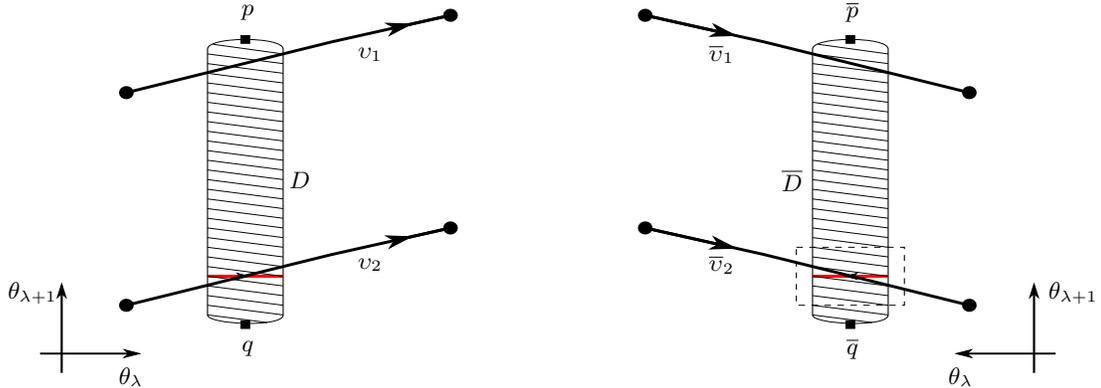
\begin{figure}
\centering
\labellist
	\small\hair 2pt
	\pinlabel $\theta_{\lambda}$ at 80 10
	\pinlabel $\theta_{\lambda+1}$ at 10 70
	\pinlabel $\theta_{\lambda}$ at 670 10
	\pinlabel $\theta_{\lambda+1}$ at 750 70
	\pinlabel $D$ at 200 150
	\pinlabel $p$ at 163 270
	\pinlabel $q$ at 163 30
	\pinlabel $\overline{D}$ at 550 150
	\pinlabel $\overline{p}$ at 593 270
	\pinlabel $\overline{q}$ at 593 30
	\pinlabel $\upsilon_1$ at 250 240
	\pinlabel $\upsilon_2$ at 250 90
	\pinlabel $\overline{\upsilon}_1$ at 500 240
	\pinlabel $\overline{\upsilon}_2$ at 500 90
\endlabellist
\includegraphics[width=.9\textwidth]{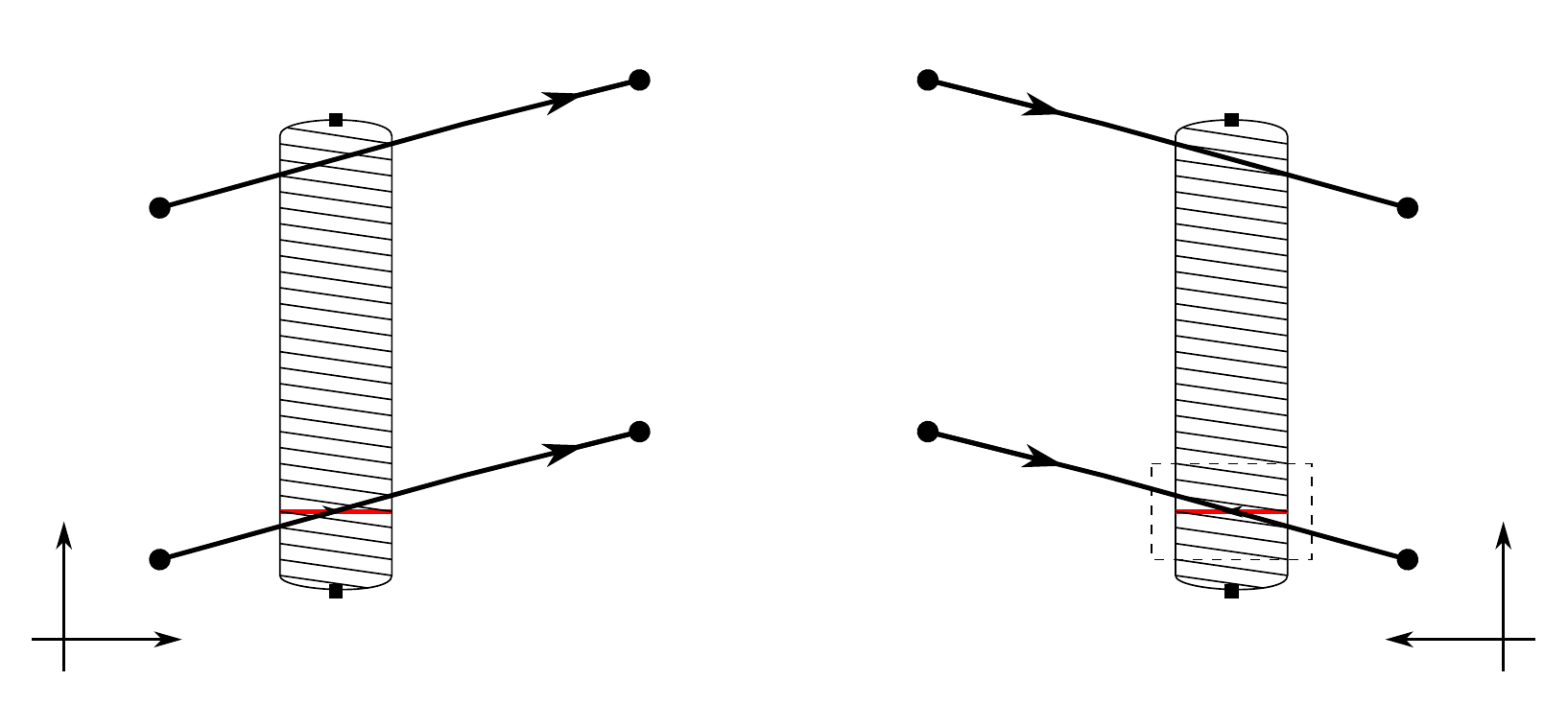}
\caption{{\it Left}: The { resolved} clasp $\upsilon_{\lambda}$ in $H_{\lambda}$. {\it Right}: The mirror { $\upsilon^r_{\lambda}$} in $-H_{\lambda}$}
\label{fig:zero-surgery-disk}
\end{figure}

By an isotopy, we can push each simple clasp into an arbitrary neighborhood of $B_{\lambda}$, where the contact structure is $C^0$-close to the horizontal foliation.  Choose a disk $D$ { in the torus $\{r_{\lambda} = \epsilon\}$} for some small $\epsilon > 0$ near the simple clasp, as in Figure \ref{fig:zero-surgery-disk}.  { Note that only the tangle $\upsilon$ is depicted, not the tangle $\tau$ or the Whitney disk as in Figure \ref{fig:compound-clasp}}.  And let $\overline{D} = \Phi_{\lambda}(D)$ be its mirror in $-H_{\lambda}$.  The characteristic foliations on $D$ and $\overline{D}$ are illustrated in Figure \ref{fig:zero-surgery-disk} as well.  Note that the contact structures are {\it not} mirrors.

Thicken $D$ to a ball $\nu(D)$ with smooth boundary.  We assume that the function $\theta_{\lambda+1}$, restricted to $\del \nu(D)$, is Morse with a single maximum at $p$ and single minimum at $q$.  Let $\nu(\overline{D}), \overline{p}, \overline{q}$ denote mirrors in $-H_{\lambda}$.  Near $\nu(D)$, the contact structure $\xi_{\lambda}$ on $H_{\lambda}$ is $C^0$-close to the foliation $\text{ker}(d \theta_{\lambda+1})$.  Thus, we can assume that the contact structure has a positive tangency to $\del \nu(D)$ near $p$, a negative tangency near $q$ and no other tangencies.  This implies that the characteristic foliation $\cF$ of $\del \nu(D)$ is the standard foliation on $S^2$, consisting of trajectories connecting these two points.  Similarly, near $\nu(\overline{D})$, the contact structure $\xi_{\lambda - 1}$ on $- H_{\lambda}$ is $C^0$-close to the foliation $- d \theta_{\lambda+1}$.  Thus, we can assume that it has a positive tangency to $\nu(\overline{D})$ near $\overline{q}$, a negative tangency near $\overline{p}$ and no other tangencies. { Again, this implies that the characteristic foliation $\overline{\cF}$ of $\del \nu (\overline{D})$ is the standard foliation on $S^2$, consisting of trajectories connecting these two points. }

The map $\Phi_{\lambda}$ restricts to an orientation-reversing diffeomorphism from $\del \nu(D)$ to $\del \nu(\overline{D})$.  Thus we can use this identification to perform topological 0-surgery.  Since the contact structures $\xi_{\lambda},\xi_{\lambda-1}$ are not mirrors, however, we need to replace $\Phi_{\lambda}$ by an isotopic map $\Phi'_{\lambda}$ that identifies the characteristic foliations.  { Nonetheless, we can control $\Phi'_{\lambda}$ in the following way: for a fixed arc $l = \{\theta_{\lambda+1} = \text{const}\}$ in $D$, we can ensure that $\Phi'_{\lambda} = \Phi_{\lambda}$ along $l$.  If we orient $l$ in the direction of $\theta_{\lambda}$, then $l$ is positively transverse to $\xi_{\lambda}$.  Moreover, the mirror arc $\overline{l}$ in $-H_{\lambda}$, again oriented in the positive $\theta_{\lambda}$ direction, is also positively transverse to $\xi_{\lambda - 1}$.  Now, extend $l$ to a simple closed curve $\gamma_+$ that separates the two singularities of the foliation $\cF$.  Provided that $l$ is short enough, we can assume that this curve is transverse to the foliation.  Similarly, we can extend $\overline{l}$ to a simple closed curve $\gamma_-$ transverse to the foliation $\overline{\cF}$. To build $\Phi'_{\lambda}$, we first declare it to equal $\Phi_{\lambda}$ along $l$, then declare that it sends $\gamma_+$ to $\gamma_-$, then extend it over the remainder of $\del \nu(D)$.}

Now, via a transverse isotopy, we can flatten each arc $\upsilon_i$ to agree with some arc $l$ along the disk $D$ and push it across $\del \nu(D) \overset{\Phi'_{\lambda}}{\simeq} \del \nu(\overline{D})$ into $-H_{\lambda}$.  It appears in $-H_{\lambda}$ as on the left of Figure \ref{fig:contact-bands}.  The symmetric band is also depicted in Figure \ref{fig:contact-bands} and attaching it is equivalent to resolving the negative crossing.  The result is the righthand side of Figure \ref{fig:contact-bands}.  We can assume the resulting arcs of $\overline{L}$ are transverse to the horizontal foliation, with is $C^0$-close to $\xi_{\lambda - 1}$, and therefore the resulting link can be assumed transverse.

{
To compute the change in self-linking number, first choose a trivialization $v$ of $\widetilde{\xi}|_{L}$.  By a homotopy, we can assume that in each local model of a simple clasp, it agrees with the vector field $\del_{r_{\lambda}} \in \text{ker}(\alpha_-)$.  In particular, this determines the blackboard framing in Figure \ref{fig:contact-bands}.  The trivialization $v$ extends to a trivialization $\widetilde{v}$ of $\widetilde{\xi}_{\widetilde{L}}$ that agrees with the vector field $\del_{r_{\lambda}}$ in this local model.  Thus, the change in the self-linking number is the same as the change in the writhe in Figure \ref{fig:contact-bands}.  The result of attaching a band is equivalent to resolving a negative crossing, so the writhe, and therefore the self-linking number, increases by 1.  Consequently, the result of attaching $2n$ bands is to increase the self-linking number by $2n$.
}
\end{proof}

\begin{figure}[h]
\centering
\labellist
	\small\hair 2pt
	\pinlabel $\upsilon_{i}$ at 80 100
	\pinlabel $\overline{\upsilon}_i$ at 300 30
	\large\hair 2pt
	\pinlabel $L$ at 100 300
	\pinlabel $\widetilde{L}$ at 1000 200
\endlabellist
\includegraphics[width=.85\textwidth]{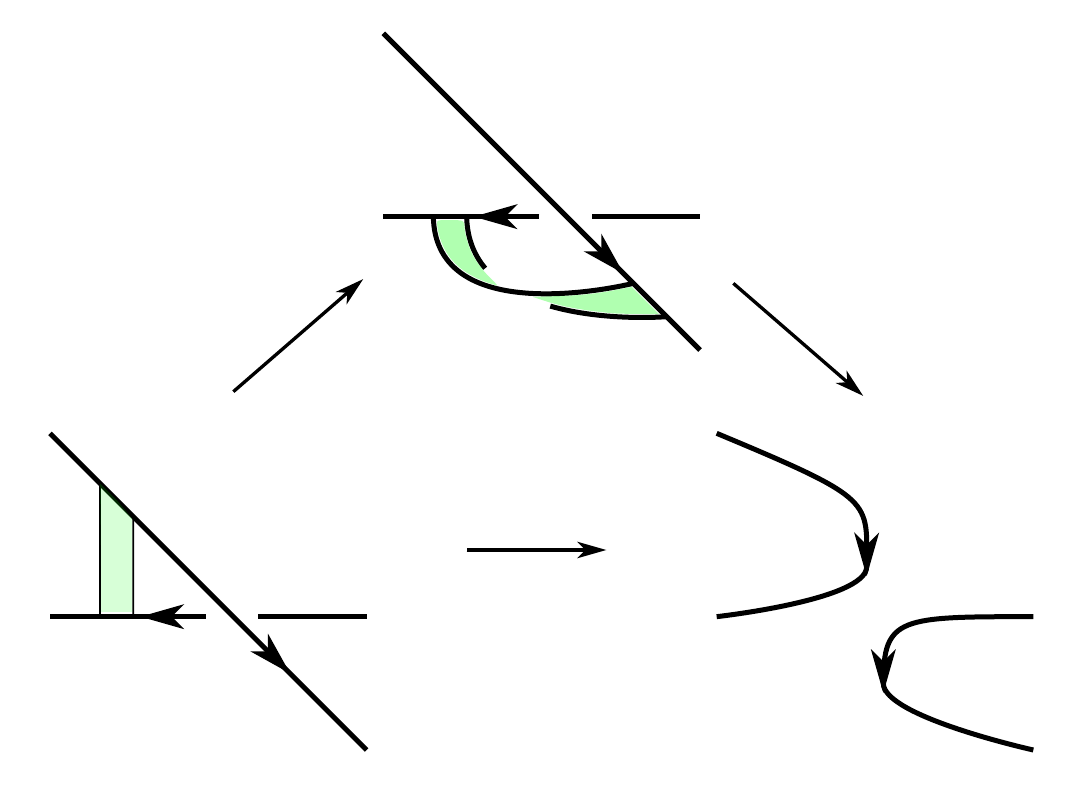}
\caption{All figures represent the interior of the dotted box on the right side of Figure \ref{fig:zero-surgery-disk}.  {\it Left}: The symmetric band to attach to $L$. {\it Middle}: The band is isotopic to a band with a single positive twist. {\it Right}: Attaching the band is equivalent to resolving the crossing in the diagram, which can clearly be made transverse to the contact structure $\xi_{\lambda+1} = \text{ker}(\alpha_-)$.}
\label{fig:contact-bands}
\end{figure}
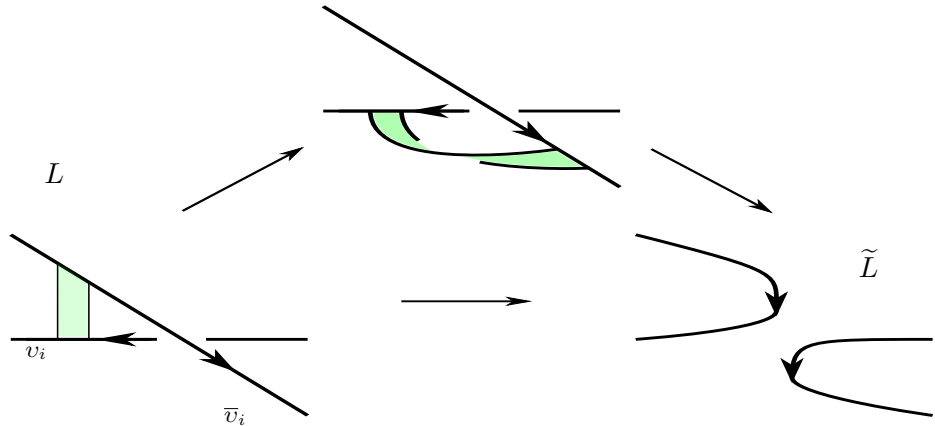

We can now prove the Thom conjecture.

\begin{proof}[Proof of Theorem \ref{thrm:Thom}]
Let $(\CP^2,\cK)$ be an embedded, oriented, connected surface of degree $d > 0$.  By Theorem \ref{thrm:MZ-GBT} we can isotope $\cK$ into bridge position and by Proposition \ref{prop:alg-trans}, we can assume that $\cK$ is in algebraically transverse bridge position.  Furthermore, by Proposition \ref{prop:simple-clasps}, we can assume that the only obstruction to geometric transversality are $n$ simple clasps.  Combining Propositions \ref{prop:Ktilde-ribbon} and \ref{prop:Ltilde-sl}, we can obtain a transverse link $\widetilde{L}$ that bounds a ribbon surface $F$ with $\chi(F) = c_1 + c_2 + c_3 - 2n$ and such that
\[sl(\widetilde{L}) = d^2 - 3d - b + 2n.\]
Applying the ribbon-Bennequin inequality (Theorem \ref{thrm:rbi-s1s2}), we see that
\[d^2 - 3d - b + 2n = sl(\widetilde{L}) \leq - \chi(F) = - c_1 - c_2 - c_3 + 2n\]
Equivalently, we have that
\[\chi(\cK) = c_1 + c_2 + c_3 - b \leq 3d - d^2\]
Solving for the genus, we have
\[g(\cK) \geq \frac{1}{2}(d-1)(d-2).\]
\end{proof}

\pagebreak

\section*{Corrigendum to {\it Bridge trisections in $\CP^2$ and the Thom conjecture}}

In \cite{LC-Thom}, an argument is given that the {\it Thom conjecture} follows from the {\it Local Thom conjecture}.  The former states that for a smoothly embedded, connected surface $\cK \subset \CP^2$ representing the class $d[\CP^1]$ with $d > 0$, then
\begin{equation}
\label{eq:Thom}
\chi(\cK) \leq 3d - d^2
\end{equation}
The Local Thom conjecture states that the slice genus of the torus knot $T(p,q)$ is $\frac{1}{2}(p-1)(q-1)$.  Rudolph showed that the Local Thom conjecture is equivalent to the {\it slice-Bennequin inequality} for transverse links in $(S^3,\xi_{std})$.  This states that
\begin{equation}
\label{eq:SBI}
sl(L) \leq -\chi(F)
\end{equation}
where $sl$ denotes the self-linking number of the transverse link $L$ and $F$ is a slice surface it bounds in $B^4$. \\

The basic strategy of the argument is, given a surface $\cK$ violating the adjunction inequality, to construct a transverse link $L$ violating the slice-Bennequin inequality.  This requires exhibiting the link $L$ and a slice surface $F$ it bounds. 

There is a fatal error concentrated in Section 6.  There are a few statements in the introduction that, in the context of the paper, logically depend on the faulty proposition.  Theorem 1.1 and Theorem 1.12 are known to be true \cite{KM-Thom}, but not proved via the arguments in the paper. In the strong sense, Theorem 1.4 is not known to be true.  However, the remaining results of Sections 1 to 5 are true.  

The key error is as follows.  In the discussion prior to Proposition 6.3, two links $\widetilde{K}$ and $\widetilde{L}$ are constructed.  Part (2) of Proposition 6.3 is correct: the link $\widetilde{K}$ bounds a surface of the given Euler characteristic.  Proposition 6.4 is also true, the link $\widetilde{L}$ is transverse with the given self-linking number.  However, it requires more care in the construction to prove Part (1) of Proposition 6.3 and show $\widetilde{K}$ and $\widetilde{L}$ are isotopic.  The construction is a compound move: the link is obtained by attaching two bands simultaneously.  Therefore, the link $\widetilde{L}$ should be constructed by {\it two} relative 1-handle attachments, followed by Legendrian surgery on a knot $U$ that runs across each of the 2-spheres geometrically once.  This has the same topological effect as the single 0-surgery in the original, while taking greater care with the contact geometry.

Consequently, Part (1) of Proposition 6.3 only holds with the following stronger assumption.

\begin{proposition}
Suppose that the knot U in admits a Legendrian realization with contact framing equal to the surface framing plus 1.  Then $\widetilde{L}$ and $\widetilde{K}$ are smoothly isotopic.
\end{proposition}

\begin{proof}
Dehn surgery can be smoothly performed on $U$ with any framing, while Legendrian Dehn surgery can only be performed with framing equal to the contact framing minus 1.  If the surgery is performed with the surface framing, then $\widetilde{L}$ and $\widetilde{K}$ are smoothly isotopic.
\end{proof}

Therefore, Theorem 1.6 holds only if for each clasp, the knot $U$ has a Legendrian realization with the correct contact framing.  There is no argument given as to why this can be accomplished. \\

%% file: Thom.bbl
\begin{thebibliography}{{Dym}04}

\bibitem[AM97]{Akbulut-Matveyev}
Selman Akbulut and Rostislav Matveyev.
\newblock Exotic structures and adjunction inequality.
\newblock {\em Turkish J. Math.}, 21(1):47--53, 1997.

\bibitem[Ben83]{Bennequin}
Daniel Bennequin.
\newblock Entrelacements et \'equations de {P}faff.
\newblock In {\em Third {S}chnepfenried geometry conference, {V}ol. 1 ({S}chnepfenried, 1982)}, volume 107 of {\em Ast\'erisque}, pages 87--161. Soc. Math. France, Paris, 1983.

\bibitem[BW84]{Boileau-Weber}
Michel Boileau and Claude Weber.
\newblock Le probl\`eme de {J}. {M}ilnor sur le nombre gordien des n\oe uds alg\'ebriques.
\newblock {\em Enseign. Math. (2)}, 30(3-4):173--222, 1984.

\bibitem[CE12]{CE-Stein}
Kai Cieliebak and Yakov Eliashberg.
\newblock {\em From {S}tein to {W}einstein and back}, volume~59 of {\em American Mathematical Society Colloquium Publications}.
\newblock American Mathematical Society, Providence, RI, 2012.
\newblock Symplectic geometry of affine complex manifolds.

\bibitem[Col97]{Colin}
Vincent Colin.
\newblock Chirurgies d'indice un et isotopies de sph\`eres dans les vari\'et\'es de contact tendues.
\newblock {\em C. R. Acad. Sci. Paris S\'er. I Math.}, 324(6):659--663, 1997.

\bibitem[{Dym}04]{Dymara}
K.~{Dymara}.
\newblock {Legendrian knots in overtwisted contact structures}.
\newblock {\em ArXiv Mathematics e-prints}, October 2004.

\bibitem[EK83]{EK-Cpoints}
Yakov Eliashberg and V.~Kharlamov.
\newblock Some remarks on the number of complex points of a real surface in a complex one.
\newblock In {\em Proceedings of Leningrad International Topology Conference, 1982)}. 1983.

\bibitem[Eli92]{Eliashberg-tight}
Yakov Eliashberg.
\newblock Contact {$3$}-manifolds twenty years since {J}. {M}artinet's work.
\newblock {\em Ann. Inst. Fourier (Grenoble)}, 42(1-2):165--192, 1992.

\bibitem[Etn99]{Etnyre-Torus-simple}
John~B. Etnyre.
\newblock Transversal torus knots.
\newblock {\em Geom. Topol.}, 3:253--268, 1999.

\bibitem[Fc03]{Forstneric-neighborhoods}
Franc Forstneri\v~c.
\newblock Stein domains in complex surfaces.
\newblock {\em J. Geom. Anal.}, 13(1):77--94, 2003.

\bibitem[Fc17]{Forstneric-book}
Franc Forstneri\v~c.
\newblock {\em Stein manifolds and holomorphic mappings}, volume~56 of {\em Ergebnisse der Mathematik und ihrer Grenzgebiete. 3. Folge. A Series of Modern Surveys in Mathematics [Results in Mathematics and Related Areas. 3rd Series. A Series of Modern Surveys in Mathematics]}.
\newblock Springer, Cham, second edition, 2017.
\newblock The homotopy principle in complex analysis.

\bibitem[FKSZ18]{FKSZ}
Peter Feller, Michael Klug, Trenton Schirmer, and Drew Zemke.
\newblock Calculating the homology and intersection form of a 4-manifold from a trisection diagram.
\newblock {\em Proceedings of the National Academy of Sciences}, 115(43):10869--10874, 2018.

\bibitem[FS95]{FS-Immersed-Thom}
Ronald Fintushel and Ronald~J. Stern.
\newblock Immersed spheres in {$4$}-manifolds and the immersed {T}hom conjecture.
\newblock {\em Turkish J. Math.}, 19(2):145--157, 1995.

\bibitem[Gei08]{Geiges}
Hansj\"org Geiges.
\newblock {\em An introduction to contact topology}, volume 109 of {\em Cambridge Studies in Advanced Mathematics}.
\newblock Cambridge University Press, Cambridge, 2008.

\bibitem[GK16]{Gay-Kirby-Trisections}
David Gay and Robion Kirby.
\newblock Trisecting 4-manifolds.
\newblock {\em Geom. Topol.}, 20(6):3097--3132, 2016.

\bibitem[KM93]{KM-Milnor}
P.~B. Kronheimer and T.~S. Mrowka.
\newblock Gauge theory for embedded surfaces. {I}.
\newblock {\em Topology}, 32(4):773--826, 1993.

\bibitem[KM94]{KM-Thom}
P.~B. Kronheimer and T.~S. Mrowka.
\newblock The genus of embedded surfaces in the projective plane.
\newblock {\em Math. Res. Lett.}, 1(6):797--808, 1994.

\bibitem[KM95]{KM-Donaldson}
P.~B. Kronheimer and T.~S. Mrowka.
\newblock Embedded surfaces and the structure of {D}onaldson's polynomial invariants.
\newblock {\em J. Differential Geom.}, 41(3):573--734, 1995.

\bibitem[KP11]{Kawamuro-Pavelescu}
Keiko Kawamuro and Elena Pavelescu.
\newblock The self-linking number in annulus and pants open book decompositions.
\newblock {\em Algebr. Geom. Topol.}, 11(1):553--585, 2011.

\bibitem[Lai72]{Lai}
Hon~Fei Lai.
\newblock Characteristic classes of real manifolds immersed in complex manifolds.
\newblock {\em Trans. Amer. Math. Soc.}, 172:1--33, 1972.

\bibitem[Law97]{Lawson-Minimal-Genus}
Terry Lawson.
\newblock The minimal genus problem.
\newblock {\em Exposition. Math.}, 15(5):385--431, 1997.

\bibitem[LC20]{LC-Thom}
Peter Lambert-Cole.
\newblock Bridge trisections in {$\Bbb{CP}^2$} and the {T}hom conjecture.
\newblock {\em Geom. Topol.}, 24(3):1571--1614, 2020.

\bibitem[LM98]{Lisca-Matic}
P.~Lisca and G.~Mati\'c.
\newblock Stein {$4$}-manifolds with boundary and contact structures.
\newblock {\em Topology Appl.}, 88(1-2):55--66, 1998.
\newblock Symplectic, contact and low-dimensional topology (Athens, GA, 1996).

\bibitem[LM18]{LC-Meier}
Peter {Lambert-Cole} and Jeffrey {Meier}.
\newblock {Bridge trisections in rational surfaces}.
\newblock {\em ArXiv e-prints}, page arXiv:1810.10450, October 2018.

\bibitem[MST96]{MST}
John~W. Morgan, Zolt\'an Szab\'o, and Clifford~Henry Taubes.
\newblock A product formula for the {S}eiberg-{W}itten invariants and the generalized {T}hom conjecture.
\newblock {\em J. Differential Geom.}, 44(4):706--788, 1996.

\bibitem[MZ17a]{MZ-Bridge-Trisection}
Jeffrey Meier and Alexander Zupan.
\newblock Bridge trisections of knotted surfaces in {$S^4$}.
\newblock {\em Trans. Amer. Math. Soc.}, 369(10):7343--7386, 2017.

\bibitem[MZ17b]{MZ-Genus-2}
Jeffrey Meier and Alexander Zupan.
\newblock Genus-two trisections are standard.
\newblock {\em Geom. Topol.}, 21(3):1583--1630, 2017.

\bibitem[MZ18]{MZ-GBT}
Jeffrey Meier and Alexander Zupan.
\newblock Bridge trisections of knotted surfaces in 4-manifolds.
\newblock {\em Proc. Natl. Acad. Sci. USA}, 115(43):10880--10886, 2018.

\bibitem[OS00a]{OS-Adjunction}
Peter Ozsv\'ath and Zolt\'an Szab\'o.
\newblock Higher type adjunction inequalities in {S}eiberg-{W}itten theory.
\newblock {\em J. Differential Geom.}, 55(3):385--440, 2000.

\bibitem[OS00b]{OS-Symplectic-Thom}
Peter Ozsv\'ath and Zolt\'an Szab\'o.
\newblock The symplectic {T}hom conjecture.
\newblock {\em Ann. of Math. (2)}, 151(1):93--124, 2000.

\bibitem[OS03]{OS-HF-Thom-1}
Peter Ozsv\'ath and Zolt\'an Szab\'o.
\newblock Absolutely graded {F}loer homologies and intersection forms for four-manifolds with boundary.
\newblock {\em Adv. Math.}, 173(2):179--261, 2003.

\bibitem[OS04]{OS-HF-Thom-2}
Peter Ozsv\'ath and Zolt\'an Szab\'o.
\newblock Holomorphic triangle invariants and the topology of symplectic four-manifolds.
\newblock {\em Duke Math. J.}, 121(1):1--34, 2004.

\bibitem[Ras10]{Rasmussen-Milnor}
Jacob Rasmussen.
\newblock Khovanov homology and the slice genus.
\newblock {\em Invent. Math.}, 182(2):419--447, 2010.

\bibitem[Rub96]{Ruberman-Minimal-Genus}
Daniel Ruberman.
\newblock The minimal genus of an embedded surface of non-negative square in a rational surface.
\newblock {\em Turkish J. Math.}, 20(1):129--133, 1996.

\bibitem[Rud83]{Rudolph-braided-ribbons}
Lee Rudolph.
\newblock Braided surfaces and {S}eifert ribbons for closed braids.
\newblock {\em Comment. Math. Helv.}, 58(1):1--37, 1983.

\bibitem[Rud93]{Rudolph-Slice-Bennequin}
Lee Rudolph.
\newblock Quasipositivity as an obstruction to sliceness.
\newblock {\em Bull. Amer. Math. Soc. (N.S.)}, 29(1):51--59, 1993.

\bibitem[Sar11]{Sarkar-Milnor}
Sucharit Sarkar.
\newblock Grid diagrams and the {O}zsv\'ath-{S}zab\'o tau-invariant.
\newblock {\em Math. Res. Lett.}, 18(6):1239--1257, 2011.

\bibitem[Shu07]{Shumakovitch}
Alexander~N. Shumakovitch.
\newblock Rasmussen invariant, slice-{B}ennequin inequality, and sliceness of knots.
\newblock {\em J. Knot Theory Ramifications}, 16(10):1403--1412, 2007.

\bibitem[Str03]{Strle}
Sa\v~so Strle.
\newblock Bounds on genus and geometric intersections from cylindrical end moduli spaces.
\newblock {\em J. Differential Geom.}, 65(3):469--511, 2003.

\bibitem[Tau94]{Taubes-SW-Symplectic}
Clifford~Henry Taubes.
\newblock The {S}eiberg-{W}itten invariants and symplectic forms.
\newblock {\em Math. Res. Lett.}, 1(6):809--822, 1994.

\bibitem[Tau95a]{Taubes-SW-Symplectic-2}
Clifford~Henry Taubes.
\newblock More constraints on symplectic forms from {S}eiberg-{W}itten invariants.
\newblock {\em Math. Res. Lett.}, 2(1):9--13, 1995.

\bibitem[Tau95b]{MR1312973}
Clifford~Henry Taubes.
\newblock More constraints on symplectic forms from {S}eiberg-{W}itten invariants.
\newblock {\em Math. Res. Lett.}, 2(1):9--13, 1995.

\bibitem[Thu86]{Thurston}
William~P. Thurston.
\newblock A norm for the homology of {$3$}-manifolds.
\newblock {\em Mem. Amer. Math. Soc.}, 59(339):i--vi and 99--130, 1986.

\end{thebibliography}
